\documentclass[12pt]%[a4paper]
{amsart}

\usepackage{geometry}
\usepackage{amsfonts}
\usepackage{amsmath}
\usepackage{amssymb}
\usepackage{amsthm,color}
\usepackage{enumerate}
\usepackage{mathtools}
 \usepackage{transparent}
 \usepackage{stackengine,tikz}
 \usepackage{wrapfig}
\usepackage{graphicx}

\usepackage{hyperref}

\usepackage{color,graphicx,enumerate,wrapfig,amssymb,mathtools}
\usepackage{epsfig}
\usepackage{pxfonts}
\usepackage{graphicx}

\usepackage{eucal}

\newcommand{\R}{{\mathbb R}}

\newcommand{\N}{{\mathbb N}}

\newcommand{\cM}{{\mathcal M}}

\newcommand{\cC}{{\mathcal C}}
\newcommand{\cE}{{\mathcal E}}

\newcommand{\bM}{{\mathbb M}}

\newcommand{\e}{\varepsilon}
\newcommand{\al}{\alpha}
\newcommand{\be}{\beta}

\newcommand{\ka}{\kappa}
\newcommand{\la}{\lambda}

\newcommand{\vp}{\varphi}

\newcommand{\dist}{\operatorname{dist}}

\newcommand{\supp}{\operatorname{supp}}
\newcommand{\D}{\nabla}
\newcommand{\p}{\partial}

\usepackage{amsmath}
\usepackage{amsfonts}
\usepackage{caption}
\usepackage{subcaption}
\usepackage{esint}

\newtheorem{theorem}{Theorem}
\theoremstyle{plain}
\newtheorem{corollary}{Corollary}
\newtheorem{definition}{Definition}

\newtheorem{lemma}{Lemma}
\newtheorem{remark}{Remark}
\newtheorem{proposition}{Proposition}

\numberwithin{equation}{section}

\makeatletter
\@namedef{subjclassname@2020}{%
  \textup{2020} Mathematics Subject Classification}
\makeatother

\begin{document}

\author{Seongmin Jeon}
\address[Seongmin Jeon]
{Department of Mathematics and Statistics \newline 
\indent University of Jyväskylä \newline 
\indent P.O. Box 35, 40014, Jyväskylä, Finland} 
\email{seongmin.s.jeon@jyu.fi}

\author{Henrik Shahgholian}
\address[Henrik Shahgholian]
{Department of Mathematics \newline 
\indent KTH Royal Institute of Technology \newline 
\indent 100 44 Stockholm, Sweden} 
\email{henriksh@kth.se}

\title[Convexity for a parabolic fully nonlinear free boundary problem]
{Convexity for a parabolic fully nonlinear free boundary problem with singular term }

\date{\today}


\thanks{S. Jeon was supported by the Academy of Finland grant 347550. H. Shahgholian was supported by Swedish Research Council grant nr 2021-03700. This work was partially done when H. Shahgholian was  spending  time at YSU (Armenia), whose support and  hospitality  is acknowledged. \\   {\it   To the memory of Marek Fila, a scientific luminary we dearly miss}}

\begin{abstract}
In this paper, we study a parabolic free boundary problem in an exterior domain
\begin{align*}
    \begin{cases}
        F(D^2u)-\partial_tu=u^a\chi_{\{u>0\}}&\text{in }(\R^n\setminus K)\times(0,\infty),\\
        u=u_0&\text{on }\{t=0\},\\
        |\D u|=u=0&\text{on }\partial\Omega\cap(\R^n\times(0,\infty)),\\
        u=1&\text{in }K\times[0,\infty).
    \end{cases}
\end{align*}
Here, $a$ belongs to the interval $(-1,0)$, $K$ is a (given) convex compact set in $\R^n$, $\Omega=\{u>0\}\supset K\times(0,\infty)$ is an unknown set, and $F$ denotes a fully nonlinear operator. Assuming a suitable condition on the initial value $u_0$, we prove the existence of a nonnegative quasiconcave solution to  the aforementioned problem, which exhibits monotone non-decreasing behavior over time.
\end{abstract}

\maketitle 

\tableofcontents

\section{Introduction}
\subsection{Background}

% \henrik{Do we need to  assume $\Omega $ is bounded?} \emph{\bblue{Do you mean we need to assume in the theorem that $\Omega$ is bounded in space? We find a solution $u$ so that $\Omega=\{u>0\}\subset B_{R_0}\times(0,\infty).$}}
% \henrik{It is not about assumption, but statement. Since we do Perron's methid, starting with supper sol Bdd in space, (uniformly in time?) then smallest supersol. should be bounded. The next step is to see how this boundedness behavies when we go in limit in the approximation.}

Let $K$ be a given  compact convex set in $\R^n$  ($n \geq 1$), $u_0 $ a nonnegative  function with compact support in $\R^n$ with $u_0 = 1$ on $K$. 
We consider the  problem of finding a solution $u$ to  
\begin{align}\label{eq:pde}
    \begin{cases}
        F(D^2u)-\partial_tu=g(u) &\text{in }(\R^n\setminus K)\times(0,\infty),\\
        u=u_0&\text{on }\{t=0\},\\
        |\D u|=u=0&\text{on }\partial\Omega\cap(\R^n\times(0,\infty)),\\
        u=1&\text{in }K\times[0,\infty),
    \end{cases}
\end{align}
where  $\Omega = \{ u >0\}$ is  an open set in $\R^{n}\times(0,\infty)$ such that $K\times(0,\infty)\subset\Omega$,
$F$ is a fully nonlinear operator, and $g(s)=s^a\chi_{\{s>0\}}$ for some $a\in(-1,0)$. Notice that $g(u)=g(u)\chi_{\{u>0\}}$.

Free boundary problems featuring singular right-hand sides like $u^a$, where $-1<a<0$, were examined by Alt-Phillips \cite{AltPhi86} for the Laplacian case. Araújo-Teixeira \cite{AraTei13} extended the analysis to problems involving fully nonlinear operators. Recently, Araújo-Sá-Urbano \cite{AraSaUrb23} addressed the singular fully nonlinear free boundary problem in the parabolic setting. Convexity configurations in parabolic partial differential equations have been extensively explored in the literature; refer to Definition \ref{Def:1} for various convexity concepts. Noteworthy studies include parabolic quasiconcavity in \cite{IshSal10,IshSal11}, space-time quasiconcavity in \cite{Bor82,CheMaSal19}, and a more general notion of parabolic power concavity in \cite{IshSal14a,IshSal14b}. The concept of a "spatially quasiconvex envelope" in the parabolic setting is discussed in \cite{KagLiuMit23}. Additional references encompass \cite{Bor96, CheMaSal19, DiaKaw93, IshSal10, IshSal11, IshSal14a, IshSal14b, KagLiuMit23, LeePetVaz06, Pet02}.

The authors, in their recent work \cite{JeoSha23}, investigated the obstacle-type convexity problem involving the fully nonlinear operator and the $p$-Laplacian. In this paper, we extend specific results from the elliptic to the parabolic setting. The main objective of this study is to show  the existence of a space-time quasiconcave solution to the singular parabolic fully nonlinear free boundary problem \eqref{eq:pde}.

\begin{definition}\label{Def:1}
    Let $u$ be a continuous function in $\R^n\times(0,\infty)$ and $\al\in(0,\infty)$ be a constant.
  \begin{enumerate}
  \item We say that $u$ is \emph{spatially quasiconcave} if 
    $$
    u((1-\rho)x_0+\rho x_1,t)\ge\min\{u(x_0,t),u(x_1,t)\}
    $$
    for every $x_0,x_1\in \R^n$, $t>0$ and $\rho\in (0,1)$.
      \item $u$ is said to be \emph{$\al$-parabolically quasiconcave} if
    $$
    u\left((1-\rho)x_0+\rho x_1,\left((1-\rho)t_0^\al+\rho t_1^\al\right)^{1/\al}\right)\ge \min\{u(x_0,t_0),u(x_1,t_1)\}
    $$
    for every $x_0,x_1\in \R^n$, $t_0,t_1>0$ and $\rho\in (0,1)$.
    \item $u$ is called to be \emph{space-time quasiconcave} if it is $1$-parabolically quasiconcave, i.e.,
    $$
    u\left((1-\rho)z_0+\rho z_1\right)\ge \min\{u(z_0),u(z
    _1)\}
    $$
    for every $z_0,z_1\in \R^n\times(0,\infty)$ and $\rho\in (0,1)$.
  \end{enumerate}  
\end{definition}

The following notion of \emph{parabolic convexity} of sets was introduced in \cite{Bor96}.

\begin{definition}
    We say that a set $E\subset \R^{n}\times[0,\infty)$ is \emph{$\al$-parabolically convex} if
    $$
    \left((1-\rho)x_0+\rho x_1, \left((1-\rho) t_0^\al+\rho t_1^\al\right)^{1/\al}\right)\in E
    $$
    whenever $(x_0,t_0), (x_1,t_1)\in E$ and $\rho\in (0,1)$.
\end{definition}

\begin{remark} It is easy to see that the $\al$-parabolically quasiconcavity of $u$ is equivalent to the space-time quasiconcavity of $\tilde u(x,t):=u(x,t^{1/\al})$. Moreover, $u$ is $\al$-parabolically quasiconcave if and only if its super-level sets $\{(x,t)\in \R^n\times[0,\infty)\,:\,u(x,t)\ge l\}$ are $\al$-parabolically convex for every $l\in\R$. Clearly, $u$ is space-time quasiconcave if and only if its super-level sets are $(n+1)$-dimensional convex sets.
\end{remark}

\subsection{Main results}
Let $S=S(n)$ be the space of $n\times n$ symmetric matrices. For constants $\Lambda\ge \lambda>0$, we let $\cM^+_{\la,\Lambda}$, $\cM^-_{\la,\Lambda}$ be the extremal Pucci operators 
$$
\cM^+_{\la,\Lambda}(M)=\Lambda\sum_{e_i>0}e_i+\la\sum_{e_i<0}e_i,\qquad \cM^-_{\la,\Lambda}(M)=\la\sum_{e_i>0}e_i+\Lambda\sum_{e_i<0}e_i,
$$
where $e_i$'s are eigenvalues of $M\in S$. We assume $F:S\to \R$ satisfies
\begin{align}
    \label{eq:assump-fully-nonlinear}
    \begin{cases}
    \text{- }F\text{ is uniformly elliptic, i.e., there are constants $\Lambda\ge\la>0$ such that}\\
    \qquad \cM^-_{\la,\Lambda}(M-N)\le F(M)-F(N)\le \cM^+_{\la,\Lambda}(M-N)\text{ for every } M,N\in S,\\
    \text{- $F$ is convex},\\
    \text{- $F$ is homogeneous of degree $1$. That is, }    \  F(rM)=rF(M) \text{ for all }   \ r>0, \ M\in S.
    \end{cases}
\end{align}

Regarding the initial value $u_0:\R^n\to\R$, we assume
\begin{align}\label{eq:assump-initial}
\begin{cases}
    \text{- $u_0$ is nonnegative, quasiconcave and compactly supported in $\R^n$},\\
    \text{- }u_0=1\quad\text{in }K,\\
    \text{- }F(D^2u_0)\ge g(u_0) \quad\text{in }\{u_0>0\}\setminus K.
\end{cases}
\end{align}

The following is the central result of this paper.

\begin{theorem}
    \label{thm:main-exterior}
    Let $K\subset\R^n$ be a compact convex set with a nonempty interior and $\al\in[1,\infty)$. Suppose that the initial data $u_0$ satisfies \eqref{eq:assump-initial}.
    Then there exists a nonnegative space-time quasiconcave function $u$, which is nondecreasing in time and satisfies \eqref{eq:pde} with $\Omega=\{u>0\}$ and
    $\{u_0>0\}\times(0,\infty)\subset\Omega\subset B_{R_0}\times(0,\infty)$ for some $R_0>0$.
\end{theorem}

It is worth noting that the solution $u$ in Theorem~\ref{thm:main-exterior} is also $\al$-parabolically quasiconcave for any $\al>1$. This follows from the time-monotonicity of $u$, and the monotonicity of $\al\longmapsto \left((1-\rho)t_0^\al+\rho t_1^\al\right)^{1/\al}$ for every $t_0,t_1>0$ and $\rho\in(0,1)$ as a consequence of Hölder's inequality.

We would like to emphasize that the quasiconcavity condition on the initial value $u_0$ in Theorem~\ref{thm:main-exterior} is not redundant. This is because the quasiconcavity of the solution implies the quasiconcavity of the initial value $u_0$ due to continuity.

In Theorem~\ref{thm:main-exterior}, we assume that $u_0$ is a subsolution of $F(D^2u_0)=g(u_0)$. It is worth noting that for general quasiconcave initial data $u_0$, even the solution of the homogeneous heat equation may not be spatially quasiconcave, as demonstrated in \cite{IshSal08}. Furthermore, \cite{ChaWei20} provides a counterexample illustrating that the additional subharmonicity assumption on $u_0$ is still insufficient to ensure the quasiconcavity of the solution.

We observe that the question of uniqueness in the free boundary problem addressed in this paper can be straightforwardly established through the application of the Lavrentiev principle when the free boundary is sufficiently smooth, for instance, $C^1$ in space and Lipschitz in time. As indicated in \cite[Remark~1]{JeoSha23}, an alternative method is necessary, even in the elliptic case, to rule out the possibility of a free boundary point with singularity. Furthermore, it can be shown  that the solution converges to its elliptic counterpart, provided the elliptic solution has a unique solution, which holds true within the class of smooth boundaries.

We would like to enumerate several open questions to which our main result, Theorem~\ref{thm:main-exterior}, can be extended:
\begin{enumerate}
\item Equations with a  right-hand side $f(u)$, and how general  can this $f(u)$ be, to allow existence of convex solutions.

%\item The case of $ a $ outside the interval $(-1,0)$: For positive values ?? the values in the range $(-3, -1)$ ?? refer to papers by De-Silva and Savin. Negtaive sign in fron of it. 

\item Asymptotic convexity, i.e., the convexity of the solution after some (uniform) time, starting from a general initial value.

\item Problems with unbounded support of initial data.

\item The case when the compact set $K$ is a "flat piece," as considered in the elliptic problem \cite{JeoSha23}.

%\item Extension to manifolds (heat/harmonic/Ricci flows).

\end{enumerate}

\subsection{Notation}
We denote the points of $\R^{n+1}$ by $z=(x,t)=(x',x_n,t)$, where $x'=(x_1,\cdots,x_{n-1})\in\R^{n-1}$. For $z^0=(x^0,t^0)\in\R^{n+1}$ and $r>0$, we denote
\begin{align*}
    &B_r(x^0):=\{x\in\R^n\,:\, |x-x^0|<r\}\,:\, \text{ ball in }\R^n,\\
    &Q_r(z^0):=B_r(x^0)\times(t^0-r^2,t^0]\,:\,\text{ parabolic cylinder in }\R^{n+1}.
\end{align*}
 We denote the gradient of $u$ by $$
\D u=D u=(\p_{x_i}u,\cdots,\p_{x_n}u).
$$
We also use the notation $D^2u$ to indicate the Hessian of $u$, representing the $n\times n$ matrix with entries $\partial_{x_i x_j}u$.
For $-1<a<0$, we consistently fix the following constant throughout the paper
$$
\be:=\frac2{1-a}\in(1,2).
$$

\section{Regularized problem}
In our problem, the presence of a highly singular right-hand side in \eqref{eq:pde} adds complexity to the investigation of the existence of quasiconcave solutions. To overcome this challenge, we employ an approach where we approximate the r.h.s. $g$ with more regular functions.

For some small constant $\e_1>0$, we consider a function $h:\R\to\R$ defined by
\begin{align}\label{eq:assump-h}
    h(s):=\begin{cases}
    0,&-\infty<s\le0,\\
    \e_1^a,&0<s<\e_1,\\
    s^a,&s\ge \e_1.
    \end{cases}
\end{align}
Notice that $h(s)=h(s)\chi_{\{s>0\}}$. Moreover, let $v_0:\R^n\to\R$ be a continuous function satisfying the first two conditions in \eqref{eq:assump-initial} and $F(D^2v_0)\ge h(v_0)$ in $\{v_0>0\}\setminus K$.

%\henrik{I am not getting a clear picture of below, 2.2. Where exactly do you solve the problem, or apply Perron's method? As you present this, it is not a FB, so somehow you need to give the domain.}
We find a nonnegative compactly supported function $v_\infty:\R^n\to \R$ satisfying
\begin{align}
    \label{eq:v-infty}
    \begin{cases}
        F(D^2v_\infty)=h(v_\infty)&\text{in }\R^n\setminus K,\\
        v_\infty=1&\text{in }K,\\
        |\D v_\infty|=0&\text{on }\partial{\{v_\infty>0\}},
    \end{cases}
\end{align}
with $\{v_\infty>0\}\supset \{v_0>0\}$ in the following way: 
for each $k\in \mathbb{N}$, we let $v_\infty^k:\R^n\to\R$ be a compactly suppport nonnegative function satisfying 
\begin{align}
    \label{eq:v-infty-k}
    \begin{cases}
        F(D^2v_\infty^k)=h(v_\infty^k)&\text{in }\R^n\setminus K,\\
        v_\infty^k=1+\frac1k&\text{in }K,\\
        |\D v_\infty^k|=0&\text{on }\partial\{v^k_\infty>0\},
    \end{cases}
\end{align}
with $v_\infty^k>v_0$ in $\{v^k_\infty>0\}$.

The existence of $\{v_\infty^k\}_{k\in\mathbb{N}}$ can be obtained by the application of the Perron's method, see the proof of \cite[Lemma~1]{JeoSha23}.
Then, over a subsequence, $v_\infty^k$ is uniformly convergent in compact sets, say to $v_\infty$, and it satisfies \eqref{eq:v-infty}. See the proof of Theorem~1 for the fully nonlinear case in \cite{JeoSha23}. \eqref{eq:v-infty-k} will play an important role in the proof of the support condition of the solution in Lemma~\ref{lem:quasiconcave-const-approx}, specifically in Step 3.

We consider the following regularized problem:

\begin{proposition}\label{prop:quasiconcave-const}
Let $K$ and $\al$ be as in Theorem~\ref{thm:main-exterior} and $h$, $v_0$ and $v_\infty$ be as above. Let $R_0>0$ be such that $\supp v_\infty\subset B_{R_0}$. Then there exists a space-time quasiconcave nonnegative function $v$ in $\R^n\times(0,\infty)$ which is nondecreasing in time and satisfies
\begin{align}
    \label{eq:sol-const}
    \begin{cases}
    F(D^2v)-\partial_tv=h(v)&\text{in } (\R^n\setminus K)\times(0,\infty),\\
    v=v_0&\text{on }\{t=0\},\\
    |\D v|=v=0&\text{on }\partial \Omega_v\cap(\R^n\times(0,\infty)),\\
    v=1&\text{in } K\times[0,\infty),
\end{cases}\end{align} 
with $\Omega_v=\{v>0\}$ satisfying 
$\{v_0>0\}\times(0,\infty)\subset \Omega_v
\subset B_{R_0}\times(0,\infty)$. Moreover, there exist constants $C>0$ and $\gamma>0$, depending only on $a, \la,\Lambda$, $n$ and $v_0$, such that for any $Q_r(z_0)\Subset (\R^n\setminus K)\times(0,\infty)$ with $z_0\in \partial\Omega_v\cap \{t>0\}$
\begin{align}
    \label{eq:reg-FN-FB}
    \sup_{Q_r(z_0)}v\le Cr^{1+\gamma}.
\end{align}
\end{proposition}

We remark that \eqref{eq:reg-FN-FB} is much stronger statement that $|\nabla v| = 0$ on $\partial \Omega_v\cap(\R^n\times(0,\infty))$.

In Section~\ref{sec:quasi-con}, we will prove Proposition~\ref{prop:quasiconcave-const} and utilize it to establish Theorem~\ref{thm:main-exterior}.

%%%%%%%%%%%%%%%%%%%%%%%%%%%%%%%%%%%%%%%%%%%%%%%%%%%%%%%%%

\section{Subsolution property of envelopes}

For a function $v$ as in Proposition~\ref{prop:quasiconcave-const},
we say that a function $v^*$ is the $\alpha$-parabolically quasiconcave envelope 
of $v$ if $v^*$ is the smallest $\alpha$-parabolically quasiconcave function that lies above $v$. The main objective of this section is to establish the subsolution property of $v^*$ (see Corollary~\ref{cor:quasi-con-subsol}), which will play a pivotal role in the proof of Proposition~\ref{prop:quasiconcave-const} in 
Section~\ref{sec:quasi-con}.

To achieve this goal, we introduce some notation. Given $\alpha\in(0,\infty)$ and $T>0$, we define $\mathcal{D}_v$ to be the $\alpha$-parabolically convex hull of the set $\{v>0\}$. In other words, $\mathcal{D}_v$ is the smallest $\alpha$-parabolically convex set that contains $\{v>0\}$. We also define $A_v:=\mathcal{D}_v\setminus(K\times[0,T])$.

For each
$$
\mu\in \bM:=\left\{\mu=(\mu_1,\mu_2,\ldots,\mu_{n+2})\in (0,1)^{n+2}\,:\, \sum_{i=1}^{n+2}\mu_i=1\right\},
$$
we set
\begin{multline*}
    A_v^{\mu}:=\Bigg\{(x,t)\in A_v\,:\, x=\sum_{i=1}^{n+2}\mu_ix_i\text{ and }t=\left(\sum_{i=1}^{n+2}\mu_it_i^\al\right)^{1/\al}\\
    \text{for some }(x_i,t_i)\in\{v>0\}\setminus (K\times[0,T]),\, 1\le i\le n+2  \Bigg\}.
\end{multline*}
For $\e\in(0,1/2)$ small, we also define
\begin{align*}
    &A_{v,\e}:=\{(x,t)\in A_v\,:\, \dist((x,t),\partial A_v\cap\{0<t<T\})>\e\},\\
    &A_{v,\e}^\mu:=\{(x,t)\in A_v^\mu\,:\, \dist((x,t),\partial A_v^\mu\cap\{0<t<T\})>\e\}.
\end{align*}

Next, for $b=(b_1,\cdots,b_{n+2})\in(0,\infty)^{n+2}$, $\mu\in\bM$ and $p\in[-\infty,\infty]$, we denote the $\mu$-weighted $p$-mean of $b$ by 
\begin{align*}
    M_p(b;\mu):=\begin{cases}
        \left[\sum_{i=1}^{n+2}\mu_ib_i^p\right]^{1/p}&\text{if }p\neq-\infty,0,\infty,\\
        \max\{b_1,\cdots,b_{n+2}\}&\text{if }p=\infty,\\
        b_1^{\mu_1}\cdots b_{n+2}^{\mu_{n+2}}&\text{if }p=0,\\
        \min\{b_1,\cdots,b_{n+2}\}&\text{if }p=-\infty.
    \end{cases}
\end{align*}
We say that a function $v$ is $\al$-parabolically $p$-concave if 
$$
v\left(\sum_{i=1}^{n+2}\mu_ix_i,\,M_\al(t_1,\cdots,t_{n+2};\mu)\right)\ge M_p\left(v(x_1,t_1),\cdots,v(x_{n+2},t_{n+2});\mu\right)
$$
for every $\mu\in\bM$ and $\{(x_i,t_i)\}_{i=1}^{n+2}$. In addition, we define the $\al$-parabolic $p$-convolution of $v$ for $\mu\in \bM$:
for $(x,t)\in A_v^\mu$
\begin{multline}\label{eq:envelope}
    V_{p,\mu}(x,t):=\sup\Bigg\{\left(\sum_{i=1}^{n+2}\mu_i v(x_i,t_i)^p\right)^{1/p}\,:\, (x_i,t_i)\in \{v>0\}\setminus (K\times[0,T]),\\
    x=\sum_{i=1}^{n+2}\mu_ix_i,\,t=\left(\sum_{i=1}^{n+2}\mu_it_i^\al\right)^{1/\al}\Bigg\}.
\end{multline}
We also set 
$$
V_p(x,t):=\sup_{\mu\in \bM}V_{p,\mu}(x,t).
$$
Notice that $V_p$ is the $\al$-parabolically $p$-concave envelope of $v$, i.e., the smallest $\al$-parabolically $p$-concave function greater than or equal to $v$.
Similarly, we define 
\begin{multline*}
    V^*_{\mu}(x,t):=\sup\Bigg\{\min\{v(x_i,t_i)\,:\, 1\le i\le n+2\} \,:\, (x_i,t_i)\in \{v>0\}\setminus (K\times[0,T]),\\
    x=\sum_{i=1}^{n+2}\mu_ix_i,\,t=\left(\sum_{i=1}^{n+2}\mu_it_i^\al\right)^{1/\al}\Bigg\}
\end{multline*}
and let 
$$
V^*(x,t):=\sup_{\mu\in \bM}V_\mu^*(x,t).
$$
Clearly, $V^*$ is the $\al$-parabolically quasiconcave envelope of $v$.

For $p\in(-\infty,0)$, let $w_p:\R^n\times[0,T^\al]\to\R$ be a function defined by
$$
w_p(x,\tau):=v^{p}(x,\tau^{1/\al}),
$$
which corresponds to (1.21) in \cite{IshLiuSal20}. Recall \eqref{eq:envelope}, and define
\begin{multline*}
    W_{p,\mu}(x,\tau):=\inf\Bigg\{\sum_{i=1}^{n+2}\mu_iw_p(x_i,\tau_i)\,:\, (x_i,\tau_i)\in\{w>0\}\setminus (K\times[0,T^\al]),\\ x=\sum_{i=1}^{n+2}\mu_ix_i,\, \tau=\sum_{i=1}^{n+2}\mu_i\tau_i\Bigg\}.
\end{multline*}
It is easy to see that
$$
W_{p,\mu}(x,\tau)=V_{p,\mu}(x,\tau^{1/\al}).
$$

In the following lemma, we prove that the supremum in \eqref{eq:envelope} is achieved at interior points $(x_i,t_i)\in(\{v>0\}\cap\{0<t\le T\})\setminus (K\times[0,T])$, $1\le i\le n+2$. This corresponds to the situation with $p<0$ in \cite[Lemma~4.1]{IshLiuSal20}, which is trivial under their assumption of zero boundary data.  However, in our case, this becomes nontrivial and requires additional technical difficulty.

\begin{lemma}
    \label{lem:envelope-int-max}
    Let $v\in C(\R^n\times(0,\infty))$ be a nonnegative function which is nondecreasing in time and satisfies \eqref{eq:sol-const}. Let $\al\in[0,\infty)$ and suppose that for every $x\in\{v_0>0\}\setminus K$
    \begin{align}
        \label{eq:initial-growth}
        \lim_{t\to0+}\frac{v^p(x,t)-v^p(x,0)}{t^\al}=-\infty.
    \end{align}
Given any $\e\in(0,1/2)$, there exists $\bar p\in(-\infty,0)$ such that for any $p\in(-\infty,\bar p)$, $\mu\in \bM$ and $(\hat x,\hat t)\in A^\mu_{v,\e}\cap\{0<t<T\}$, there exist $(x_i,t_i)\in (\{v>0\}\cap\{0<t\le T\})\setminus(K\times[0,T])$, $1\le i\le n+2$, such that
\begin{align}
    \label{eq:envelope-comb}
    \hat x=\sum_{i=1}^{n+2}\mu_ix_i,\quad \hat t=\left(\sum_{i=1}^{n+2}\mu_it_i^\al\right)^{1/\al}\quad\text{and}\quad V_{p,\mu}(\hat x,\hat t)=\left(\sum_{i=1}^{n+2}\mu_iv(x_i,t_i)^p\right)^{1/p}.
\end{align}
\end{lemma}

\begin{proof}
We divide the proof into two steps.

\medskip\noindent\emph{Step 1.} By continuity of $v$, we can find $(x_i,t_i)\in\overline{\{v>0\}\setminus (K\times[0,T])}$, $1\le i\le n+2$, such that \eqref{eq:envelope-comb} holds. We first claim that no $(x_i,t_i)$ can be contained on $\{v=0\}\cup(\{0<v_0<1\}\times\{0\})$.

To prove the claim, we assume to the contrary that $(x_i,t_i)\in\{v=0\}\cup(\{0<v_0<1\}\times\{0\})$ for some $i$, say $i=1$. If $v(x_1,t_1)=0$, then we readily have $V_{p,\mu}(\hat x,\hat t)=0$ due to the negativity of $p$. This is a contradiction since $V_{p,\mu}>0$ in $A_v^\mu$. Therefore, we may assume $0<v_0(x_1)<1$ and $t_1=0$. If $t_i=0$ for all $i$, then $\hat t=0$, which contradicts $\hat t>0$. Thus, we may assume without loss of generality $(x_2,t_2)\in(\{v>0\}\cap\{t>0\})\setminus (K\times[0,T])$. Writing $\hat\tau=\hat t^\al$ and $\tau_i=t_i^\al$, \eqref{eq:envelope-comb} reads
$$
\hat x=\sum_{i=1}^{n+2}\mu_ix_i,\quad \hat\tau=\sum_{i=1}^{n+2}\mu_i\tau_i\quad\text{and}\quad V_\mu^{p}(\hat x,\hat t)=W_{p,\mu}(\hat x,\hat\tau)=\sum_{i=1}^{n+2}\mu_i w_p(x_i,\tau_i).
$$
For small $\rho\in(0,1)$, we consider $\{(\tilde x_i, \tilde\tau_i)\}_{i=1}^{n+2}$ defined by
\begin{align*}
    &\tilde x_i=x_i\,\,\,\text{ for } 1\le i\le n+2,\\
    &\tilde\tau_1=\frac{\rho}{\mu_1},\,\,\, \tilde\tau_2=\tau_2-\frac{\rho}{\mu_2},\,\,\,\text{ and } \tilde\tau_i=\tau_i\,\,\,\text{ for } 3\le i\le n+2.
\end{align*}
Note that $\sum_{i=1}^{n+2}\mu_i\tilde\tau_i=\sum_{i=1}^{n+2}\mu_i\tau_i=\hat\tau$ and that both $(\tilde x_1,\tilde\tau_1)$ and $(\tilde x_2,\tilde\tau_2)$ are contained in $(\{w_p<\infty\}\cap\{t>0\})\setminus (K\times[0,T^\al])$ for $\rho>0$ small. The Lipschitz continuity of $w_p$ near $(\tilde x_2,\tilde\tau_2)$ gives that for some constant $C_0>0$, independent of $\rho$,
\begin{align}
    \label{eq:v-diff-est-1}
    \mu_2(w_p(\tilde x_2,\tilde\tau_2)-w_p(x_2,\tau_2))\le C_0\rho.
\end{align}
On the other hand, we observe that \eqref{eq:initial-growth} is equivalent to
$$
\lim_{t\to0+}\frac{w_p(x,t)-w_p(x,0)}t=-\infty,\quad x\in\{v_0>0\}\setminus K.
$$
This gives that 
\begin{align*}
    w_p(\tilde x_1,\tilde\tau_1)-w_p(x_1,\tau_1)=w_p\left(x_1,\frac\rho{\mu_1}\right)-w_p(x_1,0)\le -2C_0\frac{\rho}{\mu_1}
\end{align*}
if $\rho>0$ is small enough. Combining this with \eqref{eq:v-diff-est-1} yields
\begin{align*}
    \sum_{i=1}^{n+2}\mu_iw_p(\tilde x_i,\tilde\tau_i)&=\mu_1(w_p(\tilde x_1,\tilde\tau_1)-w_p(x_1,\tau_1))+\mu_2(w_p(\tilde x_2,\tilde\tau_2)-w_p(x_2,\tau_2))
    +\sum_{i=1}^{n+2}\mu_iw_p(x_i,\tau_i)\\
    &\le -C_0\rho+W_{p,\mu}(\hat x,\hat\tau).
\end{align*}
This is a contradiction by the definition of $W_{p,\mu}$.

\medskip\noindent\emph{Step 2.} In this step, we prove the lemma by making use of the result obtained in Step 1 and the idea in \cite[Lemma~5.1]{CuoSal06}. We assume to the contrary that the lemma is not true. Then we can find sequences $p_j\in(-\infty,0)$, $(\hat x_j,\hat t_j)\in A_{v,\e}^\mu\cap\{0<t\le T\}$ and $\{(x_{j,i},t_{j,i})\}_{i=1}^{n+2}\subset \{v>0\}\setminus (K\times[0,T])$ such that $p_j\to-\infty$, $(x_{j,1},t_{j,1})\in\partial K\times(0,T]$, $\hat x_j=\sum_{i=1}^{n+2}\mu_ix_{j,i}$, $\hat t_j=\left(\sum_{i=1}^{n+2}\mu_it_{j,i}^\al\right)^{1/\al}$ and $V_{p_j,\mu}(\hat x_j,\hat t_j)=\left(\sum_{i=1}^{n+2}v(x_{j,i},t_{j,i})^{p_j}\right)^{1/{p_j}}$. Over a subsequence, we have that $(\hat x_j,\hat t_j)\to(\hat x,\hat t)\in\overline{A^\mu_{v,\e}}$, $(x_{j,1},t_{j,1})\to(x_1,t_1)\in\partial K\times[0,T]$ and $(x_{j,i},t_{j,i})\to(x_i,t_i)$ for $2\le i\le n+2$. Clearly, we have $\hat x=\sum_{i=1}^{n+2}\mu_ix_i$ and $\hat t=\left(\sum_{i=1}^{n+2}\mu_it_i^\al\right)^{1/\al}$. Moreover, for each $q<0$, we have for large $j$ so that $p_j<q$
$$
V_{p_j,\mu}(\hat x_j,\hat t_j)=\left(\sum_{i=1}^{n+2}\mu_iv(x_{j,i},t_{j,i})^{p_j}\right)^{1/{p_j}}\le \left(\sum_{i=1}^{n+2}\mu_iv(x_{j,i},t_{j,i})^q\right)^{1/q}
$$
by the Jensen's inequality.
Taking $j\to\infty$ gives
$$
V_\mu^*(\hat x,\hat t)\le \left(\sum_{i=1}^{n+2}\mu_iv(x_i,t_i)^q\right)^{1/q}.
$$
We then let $q\to-\infty$ to get
$$
V_\mu^*(\hat x,\hat t)\le \min\{v(x_i,t_i)\,:\, 1\le i\le n+2\}.
$$
By the definition of $V_\mu^*$, we find
\begin{align}
    \label{eq:quasiconcave-mu}
    V_\mu^*(\hat x,\hat t)=\min\{v(x_i,t_i)\,:\, 1\le i\le n+2\}.
\end{align}
We consider the two cases either $\hat t=0$ or $\hat t>0$.

\medskip\noindent\emph{Case 1.} If $\hat t=0$, then $t_i=0$ for all $1\le i\le n+2$. Thus
$$
V_\mu^*(\hat x,0)=\min\{v_0(x_i)\,:\, 1\le i\le n+2\},
$$
with $\hat x\in \{v_0>0\}\setminus K$ and $x_1\in \partial K$. Since $V_\mu^*(\cdot,0)$ is the $\mu$-quasiconcave envelope of $v_0$, we can apply \cite[Lemma~4.2 and Remark~4.2]{CuoSal06} (see also \cite{ColSal03}) to get a contradiction. In fact, the conditions in \cite{CuoSal06} require $|\D v_0|>0$ within the domain to guarantee $0<v_0<1$ in the considered region. In our case, despite lacking $|\D v_0|>0$, we have $0<v_0<1$ in our domain $\{v_0>0\}\setminus K$. This is because $v_0$ cannot have a local maximum, as $F(D^2v_0)\ge h(v_0)>0$ in $\{v_0>0\}\setminus K$.

\medskip\noindent\emph{Case 2.} Suppose $\hat t>0$. Then $(\hat x,\hat t)$ is contained in $\overline{A^\mu_{v,\e}}\cap\{t>0\}$, i.e., it is away from $\partial_pA_v^\mu$. Using that $F(D^2v)-\partial_tv=h(v)>0$ in $\{v>0\}\setminus (K\times[0,T])$, we can follow the argument in \cite[Lemma~4.2 and Remark~4.2]{CuoSal06} to contradict \eqref{eq:quasiconcave-mu}.
\end{proof}

Next, we prove the subsolution property for the envelopes $V_p$ for small $p\in(-\infty,0)$ by following the approach in \cite{IshLiuSal20}. A similar outcome is demonstrated in \cite[Theorem~1.1]{IshLiuSal20}. However, direct application of this result to our case is not available, as \cite{IshLiuSal20} assumes a zero-Dirichlet boundary condition, whereas we are dealing with more general boundary data.

\begin{proposition}
    \label{prop:envelope-subsol}
    Let $v$ be as in Lemma~\ref{lem:envelope-int-max}. Then, for any $\al\in[1,\infty)$ and $\e\in(0,1/2)$, there exists $\bar p\in(-\infty,0)$ such that for any $p\in(-\infty,\bar p)$, $V_p$ satisfies
    $$
    F(D^2V_p)-\partial_tV_p\ge h(V_p)\quad\text{in }A_{v,\e}\cap\{0<t<T\}.
    $$
\end{proposition}

\begin{proof}[Proof of Proposition~\ref{prop:envelope-subsol}]
For $p\in(-\infty,0)$, let $G_{p,k}:(0,\infty)\times\R^n\times S(n)\to\R$ be defined by 
$$
G_{p,k}(r,\xi,X):=r^{k+\frac1p-2}F\left(-\frac{r}pX+\frac{p-1}{p^2}\xi\otimes\xi\right)+r^kh(r^{1/p}), \quad\text{where }k=1-1/p,
$$
which corresponds to (1.9) in \cite{IshLiuSal20}. We note that $w_p(x,\tau)=v^p(x,\tau^{1/\al})$ is nonincreasing in time. A direct computation shows that
\begin{align}
    \label{eq:v_p-pde}
    t^{1-\frac1\al}\partial_tw_p(x,t)+\frac{p}\al G_{p,k}(w_p,\D w_p,D^2w_p)=0.
\end{align}
We claim that the following two conditions, corresponding to (H1) and (H2) in \cite{IshLiuSal20}, respectively, are satisfied:
\begin{enumerate}
    \item[(C1)] Either $\frac1p-1+k\le0$ or $\al\left(\frac1p-1+k\right)\ge1$.
    \item[(C2)] For any $\mu\in\bM$, $r\in(0,\infty)$, $\xi\in\R^n\setminus\{0\}$ and $Y\in S(n)$, 
$$
G_{p,k}(r,\xi,Y)\le \sum_{i=1}^{n+2}\mu_iG_{p,k}(r_i,\xi,X_i)
$$
holds for $(r_i,X_i)\in(0,\infty)\times S(n)$ ($i=1,2,\cdots,n+2$) satisfying $\sum_i\mu_ir_i=r$ and 
\begin{align*}
    \text{sgn}^*(p)\left(\begin{matrix}
        \mu_1X_1& & & \\
        & \mu_2X_2 & &\\
        & & \ddots&\\
        &&& \mu_{n+2}X_{n+2}
    \end{matrix}  \right)\le \text{sgn}^*(p)\left(\begin{matrix}
        \mu_1^2Y&\mu_1\mu_2Y&\cdots&\mu_1\mu_{n+2}Y\\
        \mu_2\mu_1Y&\mu_2^2Y&\cdots&\mu_2\mu_{n+2}Y\\
        \vdots&\vdots&\ddots &\vdots\\
        \mu_{n+2}\mu_1Y&\mu_{n+2}\mu_2Y&\cdots&\mu_{n+2}^2Y
    \end{matrix} \right).
\end{align*}
\end{enumerate}
Indeed, (C1) simply holds due to our choice of $k=1-1/p$. For (C2), we observe that $r\longmapsto r^kh(r^{1/p})=\begin{cases}
\e_1^ar^k,\quad 0<r<\e_1^p,\\
r^{k+\frac ap},\quad r\ge\e_1^p\end{cases}$ is convex as $k=1-1/p>1$ and $a/p>0$. Moreover, the convexity of $F$ gives that for each fixed $\xi$, $(r,X)\longmapsto r^{k+\frac1p-2}F\left(-\frac{r}pX+\frac{p-1}{p^2}\xi\otimes\xi\right)$ is convex. Thus, $(r,X)\longmapsto G_{p,k}(r,\xi,X)$ is a convex function, and hence $G_{p,k}$ satisfies the condition (C2); see pages 348-349 in \cite{IshLiuSal20}.

Suppose the condition \eqref{eq:initial-growth} is true, which enables us to use Lemma~\ref{lem:envelope-int-max}. With this lemma, the monotonicity of $w_p$ in time and conditions (C1) and (C2) at hand, we can follow the argument in the proof of \cite[Lemma~4.2]{IshLiuSal20} to obtain that for $\bar p$ as in Lemma~\ref{lem:envelope-int-max}, $W_{p,\mu}$ is a supersolution of \eqref{eq:v_p-pde} for any $p\in(-\infty,\bar p)$. By \cite[Lemma~3.1]{IshLiuSal20}, we see that $V_{p,\mu}$ with $p\in(-\infty,\bar p)$ is a subsolution of
$$
F(D^2V_{p,\mu})-\partial_tV_{p,\mu}=h(V_{p,\mu})\quad\text{in }A_{v,\e}^\mu\cap\{0<t<T\}.
$$
This implies that $V_p=\sup_{\mu\in \bM}V_{p,\mu}$ is a subsolution.     

To close the argument, we need to prove \eqref{eq:initial-growth}. For this, we use the idea on pages 367-368 in \cite{IshLiuSal20}. For small $\tilde\e>0$, let $F_{\tilde\e}:=F+\tilde\e$ and let $v_{\tilde\e}$ be a solution of \eqref{eq:sol-const}, with $F$ replaced by $F_{\tilde\e}$. We take a nonnegative function $\psi_0\in C^2(\R^n)$ such that $\supp \psi_0\subset\supp v_0$. By continuity, we may assume $\al>1$. We take $\gamma\in(1,\al)$ and set $\psi(x,t):=v_0(x)+\psi_0(x)t^\gamma$. Then we have in $\{v_0>0\}\times(0,t_0)$ with small $t_0>0$
\begin{align*}
   F_{\tilde\e}(D^2\psi)-\partial_t\psi-h(\psi)
    & \ge F(D^2v_0+t^\gamma D^2\psi_0)+\tilde\e-\gamma t^{\gamma-1}\psi_0-h(\psi)\\
    &\ge F(D^2v_0)+t^\gamma \cM^-_{\la,\Lambda}(D^2\psi_0)+\tilde\e-\gamma r^{\gamma-1}\psi_0-h(\psi)\\
    &\ge h(v_0)-h(\psi)+t^\gamma \cM^-_{\la,\Lambda}(D^2\psi_0)-\gamma t^{\gamma-1}\psi_0+\tilde\e\\
    & \ge0,
\end{align*}
where in the last inequality we used that $h$ is nonincreasing and that $\gamma>1$. Since $\psi_0=0$ on $\partial\{v_0>0\}$, we have $v_{\tilde\e}\ge\psi$ on $\partial_p(\{v_0>0\}\times(0,t_0))$, thus the comparison principle implies $v_{\tilde\e}\ge\psi$ in $\{v_0>0\}\times(0,t_0)$. This, together with the strict inequality $\al>\gamma$, implies \eqref{eq:initial-growth} for $v_{\tilde\e}$:
\begin{align*}
    \lim_{t\to0+}\frac{v^p_{\tilde\e}(x,t)-v^p_{\tilde\e}(0,t)}{t^\al}\le \lim_{t\to0+}\frac{(v_0(x)+\psi_0(x)t^\gamma)^p-v_0^p(x)}{t^\al}=-\infty.
\end{align*}
This proves Proposition~\ref{prop:envelope-subsol} for $v_{\tilde\e}$. %By the stability theory, we can
We then pass to the limit $\tilde\e\to0$ to conclude Proposition~\ref{prop:envelope-subsol} for $v$.
\end{proof}

% \henrik{We also need derivatives to converge? Maybe in $C^{1,\alpha}$? We need gradient to be zero on the Bdry. }
% \emph{\bblue{
% For gradient to be zero on the boundary, we need to show that $v_{\tilde\e}$ satisfies \eqref{eq:reg-FN-FB}. This is somewhat trivial, as $F(D^2v_{\tilde\e})-\partial_tv_{\tilde\e}=\tilde\e+h(v_{\tilde\e})\le 2h(v_{\tilde\e})$. Do we need to mention this reasoning, or it is fine to mention nothing? I preferred not to mention it because we discuss \eqref{eq:reg-FN-FB} in the next section. Please let me know if we need to be more rigorous in this part.
% }}

\begin{lemma}
    \label{lem:p-quasi-approx}
    Let $\mathcal D$ be a bounded domain $\R^{n+1}$ and $u:\overline{\mathcal D}\to\R$ be a continuous function. For $\al\in(0,\infty)$ and $p\in(-\infty,0)$, we denote by $v_p$ the $\al$-parabolically $p$-concave envelope of $\mathcal D$ and $v^*$ the $\al$-parabolically quasiconcave envelope of $v$ in $\mathcal D$. Then $v_p\to v^*$ uniformly in $\overline{\mathcal D}$ as $p\to-\infty$.
\end{lemma}

\begin{proof}
The proof of similar to that of its elliptic counterpart \cite[Theorem~4.3]{CuoSal06}. However, we present the full proof for the sake of completeness. We note that $v_p\searrow v^*$ pointwise in $\overline{\mathcal D}$ as $p\to-\infty$. For each $p<0$, we take $\bar z_p=(\bar x_p, \bar t_p)\in\overline{\mathcal D}$ such that $v_p(\bar z_p)-v^*(\bar z_p)=\max_{z\in\overline{\mathcal D}}(v_p(z)-v^*(z))$. We also take $\mu_p=(\mu_{1,p},\cdots,\mu_{n+2,p})\in \bM$ and $\{z_{i,p}=(x_{i,p},t_{i,p})\}_{i=1}^{n+2}$ such that $\bar x_p=\sum_{i=1}^{n+2}\mu_{i,p}x_{i,p}$, $\bar t_p=\left(\sum_{i=1}^{n+2}\mu_{i,p}t_{i,p}^\al\right)^{1/\al}$ and $v_p(\bar z_p)=\left(\sum_{i=1}^{n+2}\mu_{i,p}v(z_{i,p})^p\right)^{1/p}$. Fixing $q<0$, we have that for every $p<q$
\begin{align*}
    v_p(\bar z_p)-v^*(\bar z_p)&=\left(\sum_{i=1}^{n+2}\mu_{i,p}v(z_{i,p})^p\right)^{1/p}-v^*(\bar z_p)\\
    &\le \left(\sum_{i=1}^{n+2}\mu_{i,p}v(z_{i,p})^q\right)^{1/q}-\min\{v(z_{1,p}),\cdots, v(z_{n+2,p})\}.
\end{align*}
Over a subsequence, $\mu_p\to \mu=(\mu_1,\cdots,\mu_{n+2})\in \bM$ and $z_{i,p}\to z_i\in\overline{\mathcal D}$ as $p\to-\infty$. It then follows that
\begin{align*}
    \limsup_{p\to-\infty}\max_{z\in\overline{\mathcal D}}|v_p(z)-v^*(z)|&=\limsup_{p\to-\infty}(v_p(\bar z_p)-v^*(\bar z_p))\\
    &\le\left(\sum_{i=1}^{n+2}\mu_iv(z_i)^q\right)^{1/q}-\min\{v(z_1),\cdots,v(z_{n+2})\}.
\end{align*}
Finally, by taking $q\to-\infty$, we conclude the lemma.    
\end{proof}

From Proposition~\ref{prop:envelope-subsol} and Lemma~\ref{lem:p-quasi-approx}, we immediately obtain the following subsolution property of the $\al$-parabolically quasiconcave envelope of $v$.

\begin{corollary}
    \label{cor:quasi-con-subsol}
    Let $v$ be as in Lemma~\ref{lem:envelope-int-max}. Then, for any $\al\ge1$, the $\al$-parabolically quasiconcave envelope $v^*$ of $v$ satisfies
    $$
    F(D^2v^*)-\partial_tv^* \geq h(v^*)\quad\text{in }\{v^*>0\}\setminus (K\times[0,T]).
    $$
\end{corollary}

%\henrik{Maybe bettwr to say: Let $v$ be as in Lemma~\ref{lem:envelope-int-max}. Then, for any $\al\ge1$, the $\al$-parabolically quasiconcave envelope $v^*$ of $v$ satisfies
%    $$
%    F(D^2v^*)-\partial_tv^* \geq h(v^*)\quad\text{in }\{v^*>0\}\setminus (K\times[0,T]).
%    $$}

%%%%%%%%%%%%%%%%%%%%%%%%%%%%%%%%%%%%%%%%%%%%%%%%%%%%

\section{Existence of a quasiconcave solution}\label{sec:quasi-con}
The objective of this section is to prove Proposition~\ref{prop:quasiconcave-const} and utilize it to establish Theorem~\ref{thm:main-exterior}. To prove Proposition~\ref{prop:quasiconcave-const}, we need some auxiliary results, Lemma~\ref{lem:quasiconcave-const-approx} and Lemma~\ref{lem:sol-supp}. 

%\henrik{There is a tautology here, related to $R_0$. In Lemma 1 you saye there is a $R_0$ etc.. and no in Lema 4 you say let $R_0$ be  as in Lemma1. When I go to proof of Lemma 1, there is no mention of $R_0$, and you refer to Lemma 4. We need to clarify this. }
%\emph{\bblue{In Theorem~\ref{thm:main-exterior}, Lemma~\ref{lem:quasiconcave-const} and Lemma~\ref{lem:quasiconcave-const-approx}, $R_0$ is any positive real number such that $B_{R_0}$ contains the support of the initial value, e.g., $\overline{\{v_0>0\}}\subset B_{R_0}$. I made a modification in the statement of Lemma~\ref{lem:quasiconcave-const} in blue text. }}

\begin{lemma}\label{lem:quasiconcave-const-approx}
Let $h$, $v_0$, $K$ and $R_0$ be as in Proposition~\ref{prop:quasiconcave-const}. Then there exists a nonnegative function $v$ in $B_{R_0}\times(0,\infty)$ that satisfies
\begin{align}
    \label{eq:sol-const-approx}
    \begin{cases}
    F(D^2v)-\partial_tv=h(v) &\text{in } (B_{R_0}\setminus K)\times(0,\infty),\\
    v=v_0&\text{on }\{t=0\},\\
    v=1&\text{in } K\times[0,\infty),\\
    v=0&\text{on }\partial B_{R_0}\times(0,\infty).
\end{cases}\end{align} 
%\henrik{It seems we have missed showing the growth/decay condition 2.5, in the proof.}
%\emph{\bblue{I highlighted in brown in the next page regarding the condition 2.5. This follows by applying \cite[Corollary~1]{AraSaUrb23}.}}

Moreover, $v$ satisfies the growth condition \eqref{eq:reg-FN-FB} and the support condition $\{v_0>0\}\times [0,\infty)\subset \supp v \subset B_{R_0}\times[0,\infty)$, and it is nondecreasing in time-variable in $B_{R_0}\times(0,\infty)$.
\end{lemma}
%\henrik{Here above you somehow indicate  that the support of sol is in the ball, in space. Also in the proof a few line below, the same indication is made. }

\begin{proof}
We split the proof into several steps.

\medskip\noindent\emph{Step 1 (Existence \& growth condition):}
For $\e_1>0$ as in \eqref{eq:assump-h}, we consider two functions $v_\sharp$ and $v^\sharp$, satisfying 
\begin{align}\label{eq:subsol-perron}
    \begin{cases}
        F(D^2v_\sharp)-\partial_tv_\sharp=\e_1^a\chi_{\{v_\sharp> 0 \}} &\text{in }B_{R_0}\times(0,\infty),\\
        v_\sharp=v_0&\text{on }\{t=0\},\\
        v_\sharp=1&\text{in }K\times[0,\infty),\\
        v_\sharp=0&\text{on }\partial B_{R_0}\times(0,\infty),
    \end{cases}
\end{align}
and
\begin{align}\label{eq:supersol}
    \begin{cases}
        F(D^2v^\sharp)-\partial_tv^\sharp=0&\text{in }B_{R_0}\times(0,\infty),\\
        v^\sharp=v_0&\text{on }\{t=0\},\\
        v^\sharp=1&\text{in }K\times[0,\infty),\\
        v^\sharp=0&\text{on }\partial B_{R_0}\times(0,\infty),
    \end{cases}
\end{align}
respectively. We note that their existence follows from the standard Perron's method (smallest super-solution) and $v^\sharp\ge0$ in $B_{R_0}\times(0,\infty)$ due to the minimum principle. We claim that
\begin{align}\label{eq:subsol}
v_\sharp\ge 0\,\,\,\text{in}\,\,\, B_{R_0}\times(0,\infty).
\end{align}
To prove it by contradiction, we assume to the contrary for some $T\in(0,\infty)$
$$
A:=\{z\in(B_{R_0}\setminus K)\times[0,T)\,:\, v_\sharp(z)<0\}\neq\emptyset.
$$
Clearly, $F(D^2v_\sharp)-\partial_tv_\sharp=0$ in $A$. In addition, since $v_\sharp\ge0$ on $\partial_p((B_{R_0}\setminus K)\times[0,T))$, we have $v_\sharp=0$ on $\partial_pA$. Those violate the minimum principle.

To find a solution of \eqref{eq:sol-const-approx} by making use of $v_\sharp$ and $v^\sharp$, we use the penalization method adopted in \cite{AraSaUrb23}. Let $\vp\in C^\infty(\R)$ be compactly supported in $[0,1]$ with $\int_0^1\vp(\rho)d\rho=1$. For $\beta=\frac2{1-a}\in (1,2)$ and a parameter $0<\sigma_0<1$, define for each $j\in \mathbb{N}$
$$
\mathcal{B}_j(s):=\int_0^{\frac{s-\sigma_0j^{-\beta}}{j^{-\beta}}}\vp(\rho)d\rho,\quad s\in \R,
$$
and notice that $\mathcal{B}_j$ is an approximation of the characteristic function $\chi_{(0,\infty)}$. For a Lipschitz function $h_j:=\mathcal{B}_j h$, we then consider the problem 
% \henrik{What does $\mathcal{B}_j h$ mean? The product does not mean much, since $\mathcal{B}_j$ is defined on real line. If $\mathcal{B}_j (h)$, then it has to be spelled out.} \emph{\bblue{
% $h$ has discontinuity at $0$. So we approximate $h$ by $\{\mathcal{B}_j h\}_j$, which is Lipschitz in $\R$. $\mathcal{B}_j h$ is Lipschitz across $0$ since $\mathcal{B}_j(s)=0$ for $s<\sigma_0j^{-\be}$.
% }}
\begin{align}
    \label{eq:sol-const-reg}
    \begin{cases}
    F(D^2v^j)-\partial_tv^j=h_j(v^j)&\text{in }(B_{R_0}\setminus K)\times(0,\infty),\\
    v^j=v_0&\text{on }\{t=0\},\\
    v^j=1&\text{in } K\times[0,\infty),\\
    v^j=0&\text{in }\partial B_{R_0}\times(0,\infty).
    \end{cases}
\end{align}
Clearly, $v_\sharp$ and $v^\sharp$ are sub- respectively   supersolution to
\begin{align}
    \label{eq:sol-const-reg-pde}
    F(D^2w)-\partial_tw=h_j(w)\quad\text{in }(B_{R_0}\setminus K)\times(0,\infty).
\end{align}
Let
\begin{align*}
    \mathcal{S}^j:=\{w\in C(B_{R_0}\times [0,\infty))\,:\,v_\sharp\le w\le v^\sharp\text{ and }w \text{ is a supersolution of }\eqref{eq:sol-const-reg-pde}\}.
\end{align*}
Then,
\begin{align}
    \label{sol-const-reg}
    v^j:=\inf_{w\in \mathcal{S}^j}w
\end{align}
is a solution of \eqref{eq:sol-const-reg-pde}, see e.g. \cite[Theorem~3.1]{RicTeyUrb19}. 
From $v_\sharp\le v^j\le v^\sharp$, it is evident that $v^j$ satisfies \eqref{eq:sol-const-reg}. Furthermore, utilizing $0\le h_j(v^j)\le (v^j)^a\chi_{\{v^j>0\}}$, we can easily verify that the argument supporting the regularity of solutions to the regularized problem in \cite{AraSaUrb23} is applicable in our context. Since $v^j\le 1$ by the maximum principle, there exists a constant $C>0$, independent of $j$, such that
%\henrik{Below: If $z_0\in \partial\{v^j>0\}\cap\{t>0\}$ , then obviously  $v^j(z_0) = 0$ }
\begin{align*}
    \sup_{Q_r(z_0)}v^j\le Cr^{1+\gamma}
\end{align*}
whenever $z_0\in \partial\{v^j>0\}\cap\{t>0\}$ and $Q_r(z_0)\Subset (B_{R_0}\setminus K)\times(0,\infty)$. By the Arzelà-Ascoli Theorem, there exists a continuous function $v$ in $(B_{R_0}\setminus K)\times(0,\infty)$ such that, along a subsequence, $v^j\to v$ uniformly on compact subsets of $(B_{R_0}\setminus K)\times(0,\infty)$. Letting $v=1$ in $K\times(0,\infty)$ and considering the uniform convergence and the fact that $v_\sharp\le v^j\le v^\sharp$, we deduce that $v$ satisfies \eqref{eq:sol-const-approx} and the growth condition \eqref{eq:reg-FN-FB}.

\medskip\noindent\emph{Step 2 (Monotonicity):}
In this step, we establish that each solution $v^j$ of \eqref{eq:sol-const-reg} is nondecreasing in the time variable within $B_{R_0}\times(0,\infty)$. Notably, this property readily extends to $v$ due to the uniform convergence $v^j\to v$. To demonstrate the monotonicity of $v^j$, we extend $v_0$ from $\R^n$ to $\R^n\times[0,\infty)$ by

$$
v_0(x,t):=v_0(x).
$$
We first claim that
\begin{align}
    \label{eq:v-v_0-ineq}
    v^j\ge v_0\quad\text{in }\{v_0>0\}\times(0,\infty).
\end{align}
%\henrik{Same question about $\mathcal{B}_jh$, here below.}\emph{\bblue{The argument below works when $h_j$ is Lipschitz in $\R$, or $[0,1]$. $h$ is Lipchitz only in $(0,\infty)$}}

To prove the claim, we assume to the contrary that for some $T>0$, $A_0:=\{v^j<v_0\}\cap((\{v_0>0\}\setminus K)\times(0,T])$ is nonempty. Since $v^j\ge v_0$ on $\partial_p((\{v_0>0\}\setminus K)\times(0,T])$, we have $v^j=v_0$ on $\partial_pA_0$. Moreover, observe that $h_j=\mathcal{B}_jh\le h$ as $0\le\mathcal{B}_j\le1$. This, along with the Lipschitz regularity of $h_j$, gives that for some constant $c_j>0$, we have in $A_0$
\begin{align*}
    \cM^+_{\la,\Lambda}(D^2(v_0-v^j))-\partial_t(v_0-v^j)&\ge (F(D^2v_0)-\partial_tv_0)-(F(D^2v^j)-\partial_tv^j)\\
    &\ge h(v_0)-h_j(v^j)\ge h_j(v_0)-h_j(v^j)\\
    &\ge -c_j(v_0-v^j).
\end{align*}
Now consider the function  
$$
w(x,t):=e^{-c_jt}(v_0-v^j)(x,t)
$$
which satisfies  $w>0$ in $A_0$ and $w=0$ on $\partial_pA_0$. Moreover,  in $A_0$
\begin{align*}
    \cM^+_{\la,\Lambda}(D^2w)-\partial_tw=e^{-c_jt}\left(\cM^+_{\la,\Lambda}(D^2(v_0-v^j))-\partial_t(v_0-v^j)\right)+c_je^{-c_jt}(v_0-v^j)\ge 0.
\end{align*}
These properties contradict the maximum principle, and hence \eqref{eq:v-v_0-ineq} is proved.

Next, we consider translations of $v^j$ in time: for $\rho>0$
$$
v^j_\rho(x,t):=v^j(x,t+\rho),\quad (x,t)\in B_{R_0}\times(0,\infty).
$$
We claim that for every $\rho>0$
\begin{align}
    \label{eq:v-trans-ineq}
    v^j\le v^j_\rho\quad\text{in }B_{R_0}\times(0,\infty).
\end{align}
Notice that this assertion immediately implies the nondecreasing nature of $v^j$ in time within $B_{R_0}\times(0,\infty)$, thereby completing the proof. To establish \eqref{eq:v-trans-ineq}, we assume, for contradiction, that there exist $\rho>0$ and $T>0$ such that $A_\rho:={v^j>v^j_\rho}\cap\left((B_{R_0}\setminus K)\times(0,T)\right)\neq\emptyset$. Since $v^j_\rho\ge v^j$ on $\partial_p\left((B_{R_0}\setminus K)\times(0,T)\right)$ due to \eqref{eq:v-v_0-ineq}, it follows that $v_\rho^j=v^j$ on $\partial_pA_\rho$. Moreover, by using the Lipschitz continuity of $h_j$, we have in $A_\rho$
\begin{align*}
    \cM^+_{\la,\Lambda}(D^2(v^j-v^j_\rho))-\partial_t(v^j-v^j_\rho)&\ge (F(D^2v^j)-\partial_tv^j)-(F(D^2v^j_\rho)-\partial_tv^j_\rho)\\
    &=h_j(v^j)-h_j(v^j_\rho)\ge-c_j(v^j-v^j_\rho).
\end{align*}
Then, for
$$
w_\rho(x,t):=e^{-c_jt}(v^j-v^j_\rho)(x,t),
$$
we can argue as above to get
\begin{align*}
\begin{cases}
    w_\rho>0\,\,\,,\,\,\, \cM_{\la,\Lambda}^+(D^2w_\rho)-\partial_tw_\rho\ge 0&\text{in }A_\rho,\\
    w_\rho=0&\text{on }\partial_p A_\rho.
\end{cases}\end{align*}
This contradicts the maximum principle.

\medskip\noindent\emph{Step 3 (Support condition):}
In this step, we show $\{v_0>0\}\times[0,\infty)\subset \supp v\subset B_{R_0}\times[0,\infty)$. The first inclusion simply follows from the monotonicity of $v$ in time. For the second inclusion, it is sufficient to prove
\begin{align}
    \label{eq:v-v-inf-comp}
    v(x,t)\le v_\infty(x)\quad\text{for all }(x,t)\in\{x\in\R^n\,:\, v_\infty(x)>0\}\times(0,\infty).
\end{align}
To prove \eqref{eq:v-v-inf-comp}, let $\{v_\infty^k\}_{k\in\mathbb{N}}$ be as in \eqref{eq:v-infty-k}. We extend $v_\infty$ and $v_\infty^k$ from $\R^n$ to $\R^n\times[0,\infty)$ by 
$$
v_\infty(x,t):=v_\infty(x),\quad v_\infty^k(x,t):=v_\infty^k(x).
$$
For small $\e>0$, we write $A_{k,\e}:=\{x\in\R^n\,:\, v_\infty^k(x)>\e\}$. We claim
\begin{align}
    \label{eq:v-infty-k-claim}
    v_\infty^k>v\quad\text{in }A_{k,\e}\times(0,\infty).
\end{align}
Indeed, since $v_\infty^k>v_0=v$ on the compact set $\overline{A_{k,\e}}\times\{0\}$, we have by continuity
\begin{align}
    \label{eq:v-infty-k-v}
    v_\infty^k>v\quad\text{in }\overline{A_{k,\e}}\times[0,T_0)\quad\text{for some }T_0>0.
\end{align}
Suppose the claim \eqref{eq:v-infty-k-claim} is not true. Then, due to \eqref{eq:v-infty-k-v}, we can find a point $(x^0,t^0)\in (A_{k,\e}\setminus K)\times(0,\infty)$ such that $v_\infty^k=v>0$ at $(x^0,t^0)$ and $v_\infty^k>v$ in $A_{k,\e}\times(0,t^0)$. In particular, we have $v_\infty^k-v\ge0$ in an open set $D:=((A_{k,\e}\setminus K)\times(0,t^0))\cap \{v>0\}$. Moreover, using the monotonicity of $h$, we obtain that in $D$
\begin{align*}
    \cM_{\la,\Lambda}^-(D^2(v_\infty^k-v))-\partial_t(v_\infty^k-v)&\le (F(D^2v_\infty^k)-\partial_tv_\infty^k)-(F(D^2v)-\partial_tv)\\
    &=h(v_\infty^k)-h(v)\le0.
\end{align*}
Then, the strong minimum principle implies that $v_\infty^k=v$ in $D$. This contradicts \eqref{eq:v-infty-k-v}, and hence the claim \eqref{eq:v-infty-k-claim} is proved.

Since $\e>0$ is arbitrary, \eqref{eq:v-infty-k-claim} gives that $v_\infty^k>v$ in $\{x\in\R^n\,:\, v_\infty^k(x)>0\}\times(0,\infty)$. Taking $k\to\infty$, we conclude \eqref{eq:v-v-inf-comp}.
\end{proof}

\begin{lemma}
    \label{lem:sol-supp}
    Let $v\in C(\R^n\times(0,\infty))$ be a nonnegative function which is nondecreasing in time and satisfies \eqref{eq:sol-const}. Then $$
    \supp v\cap\{t=0\}=\supp v_0.
    $$
\end{lemma}

\begin{proof}
%\henrik{This is correct, but what about an alternative way, like this: We claim each $t$-section of the support is connected. If not, then due to the fact that sol is increasing, we have a component of this  t-section ($t=t_1$) $A_1$ that does not see support of $v_0$. Again by monotonicity, $A_1 \times (0,t_1)$ does not see support of $v_0$ (when consider a cylider over this supprot). Hence maximum principle can now apply to our function in $A_1 \times (0,t_1)$ to reach a contra. }
We first claim that each time section of $\supp v$ is connected. Assuming the contrary, suppose the claim is not true. Then, due to the monotonicity of $v$, we can find $t_1>0$ such that a connected component of a time section $\supp v\cap\{t=t_1\}$, say $A_1\times\{t_1\}$, satisfies $A_1\cap\supp v_0=\emptyset$. For $\mathbb{A}_1:=(\text{Int }A_1\times(0,t_1))\cap\{v>0\}$, this implies that $v=0$ on $\partial_p\mathbb A_1$. On the other hand, $v>0$ and $F(D^2v)-\partial_tv=h(v)>0$ in $\mathbb A_1$. Those violate the maximum principle.

Now, we observe that the monotonicity of $v$ gives $\supp v\cap\{t=0\}\supset \supp v_0$. To prove the reverse inclusion, we assume to the contrary that $A_2:=\text{Int }(\supp v\cap\{t=0\})\setminus \supp v_0$ is a nonempty open set in $\R^n$. Notice that $A_2$ is connected due to the claim above. By the monotonicity of $v$, we have $v>0$ in $A_2\times(0,\infty)$. Since $v=0$ on $\overline{A_2}\times\{0\}$, we can find a small constant $t_2>0$ such that $v<\e_1$ in $A_2\times(0,t_2)$, where $\e_1>0$ is as in \eqref{eq:assump-h}. Then
$$
F(D^2v)=\partial_tv+h(v)\ge \e_1^a\quad\text{in }A_2\times(0,t_2).
$$
By the nondegeneracy, there exists a constant $c_0>0$ such that
$$
\sup_{A_2\times\{s\}}v\ge c_0\quad\text{for every }0<s<t_2.
$$
Taking $s\to0$ gives $\sup_{A_2}v_0\ge c_0$, which is a contradiction since $v_0=0$ in $A_2$.  
\end{proof}

Now we prove Proposition~\ref{prop:quasiconcave-const} by making use of Lemmas~\ref{lem:quasiconcave-const-approx} and ~\ref{lem:sol-supp}.

\begin{proof}[Proof of Proposition~\ref{prop:quasiconcave-const}]
Let $v\in C(B_{R_0}\times(0,\infty))$ be a function as in Lemma~\ref{lem:quasiconcave-const-approx}. We extend $v$ to $\R^n\times(0,\infty)$ by letting $v\equiv 0$ in $(\R^n\setminus B_{R_0})\times(0,\infty)$. Then, from the support condition $\supp v\subset B_{R_0}\times(0,\infty)$ and the growth condition \eqref{eq:reg-FN-FB}, we see that $v$ satisfies \eqref{eq:sol-const}. In view of Lemma~\ref{lem:quasiconcave-const-approx}, it remains to prove the space-time quasiconcavity of $v$. To this aim, let $v^*$ be the space-time quasiconcave envelope of $v$, which, by Corollary~\ref{cor:quasi-con-subsol}, is a subsolution of
$$
F(D^2v^*)-\partial_tv^*= h(v^*)\quad\text{in }\{v^*>0\}\setminus(K\times(0,\infty)).
$$

As $v^*$ is space-time quasiconcave and $v^*\ge v$, it is sufficient to prove $v\ge v^*$. We will use the Lavrentiev Priniciple to achieve this. We assume without loss of generality that the origin $0$ is contained in $K^{\mathrm{o}}$, the interior of $K$. We define for $T>1$ and $0<\rho<1$
\begin{align*}
    &D_T:=\{(x,t)\in \R^n\times(0,T]\,:\, v(x,t)>0\},\quad D_T^*:=\{(x,t)\in \R^n\times(0,T]\,:\, v^*(x,t)>0\},\\
    &v^\rho(x,t):=v(\rho x,\rho^2t),\quad D_T^\rho:=\{(x,t)\in \R^n\times(0,T]\,:\, v^\rho(x,t)>0\},
\end{align*}
and let 
$$
E_T:=\{0<\rho<1\,:\,v^\rho\ge v^*\,\,\,\text{in }D_T^*\}.
$$
Notice that both $D_T$ and $D_T^*$ are bounded since $v$ restricted to $\R^n\times[0,T]$ has a compact support. As $v_0:\R^n\to\R$ is a nonnegative compactly supported function satisfying $F(D^2v_0)\ge h(v_0)>0$ in $\{v_0>0\}\setminus K$ and $v_0=1$ in $K$, the maximum principle gives $0\le v_0\le1$ in $\R^n$. This in turn implies by the maximum principle $0\le v\le1$ in $\R^n\times(0,\infty)$, and hence $0\le v^*\le1$ in $\R^n\times(0,\infty)$. On the other hand, the assumption $0\in K^{\mathrm{o}}$ guarantees that $D_T^*\subset(\rho^{-1}K)\times(0,T]$ for small $\rho>0$. Since $v^\rho=1$ in $(\rho^{-1}K)\times(0,T]$, we have $v^\rho\ge v^*$ in $\R^n\times(0,T]$ for such $\rho>0$, and hence $E_T\neq\emptyset$. Note that our desired inequality $v\ge v^*$ in $\R^n\times(0,\infty)$ follows once we show that $\sup E_T=1$ for every $T>1$. Now, we assume towards a contradiction that for some $T>1$, $\rho_0:=\sup E_T\in (0,1)$.

We claim that there is a point $z^0\in \overline{D_T^*}\setminus (K\times[0,T])$ such that $v^{\rho_0}(z^0)=v^*(z^0)$. To prove it by contradiction, suppose that $v^{\rho_0}>v^*$ in $\overline{D_T^*}\setminus (K\times[0,T])$. For a fixed $\tilde\rho_0\in (\rho_0,1)$, we have $v^{\rho_0}>v^*$ in a smaller compact set $\overline{D_T^*}\setminus (\tilde\rho_0^{-1}K^{\mathrm{o}}\times[0,T])$. By continuity, there exists $\rho_1\in (\rho_0,\tilde\rho_0)$ such that $v^{\rho_1}\ge v^*$ in $\overline{D_T^*}\setminus(\tilde\rho_0^{-1}K^{\mathrm{o}}\times[0,T])$. Since $v^{\rho_1}=1\ge v^*$ in $\rho_1^{-1}K\times[0,T]$ and $\rho_1^{-1}K\times[0,T]\supset \tilde\rho_0^{-1}K^{\mathrm{o}}\times[0,T]$, we have $v^{\rho_1}\ge v^*$ in $\overline{D_T^*}$, contradicting $\rho_0=\sup E_T$.

Now we consider the following three possibilities.
\begin{align*}
    &\mathbf{A}:\quad z^0\in \left(\overline{D_T^*}\setminus (K\times[0,T])\right)\cap \{t=0\},\\
    &\mathbf{B}:\quad z^0\in \left(\overline{D_T^*}\setminus (K\times[0,T])\right)\cap \{t>0\}\,\,\,\text{and}\,\,\, v^*(z^0)=v^{\rho_0}(z^0)>0 ,\\
    &\mathbf{C}: \quad z^0\in \left(\overline{D_T^*}\setminus (K\times[0,T])\right)\cap \{t>0\}\,\,\,\text{and}\,\,\, v^*(z^0)=v^{\rho_0}(z^0)=0.
\end{align*}

\noindent\emph{Case} $\mathbf{A}$. We first consider the case $z^0=(x^0,0)\in \{t=0\}$. Using the fact that $\tilde v^*(x,t):=v^*(x,t^{1/\al})$ is space-time quasiconcave in $\R^n\times(0,\infty)$, it is easy to see that $v^*(\cdot,0)$ is the quasiconcave envelope of $v(\cdot,0)$ in $\R^n$. This, along with the quasiconcavity assumption on $v_0$, implies that $v^*=v_0=v$ on $\{t=0\}$. We further split into two cases
$$
\text{either}\quad v^*(z^0)=v^{\rho_0}(z^0)=0\quad \text{or}\quad v^*(z^0)=v^{\rho_0}(z^0)>0.
$$

\medskip\noindent\emph{Case} $\mathbf{A.1}$. Suppose $v^*(x^0,0)=v^{\rho_0}(x^0,0)=0$. Note that this is equivalent to $v_0(x^0)=v_0(\rho_0x^0)=0$. Since $\supp v^*$ is the convex hull of $\supp v$, we have by Lemma~\ref{lem:sol-supp} that $\supp v^*\cap\{t=0\}=\supp v_0$. Thus, it follows that both $x^0$ and $\rho_0x^0$ are contained on $\partial\{v_0>0\}$. This is a contradiction since $\{v_0>0\}$ is a convex set containing the origin.

\medskip\noindent\emph{Case} $\mathbf{A.2}$. Next, we suppose $\al_0:=v^*(x^0,0)=v^{\rho_0}(x^0,0)>0$. Since $v_0(x^0)=v_0(\rho_0x^0)=\al_0$ and $v_0$ is quasiconcave, we have $v_0=\al_0$ on the line $l:=[\rho_0x^0,x^0]$. This, combined with the fact that both $\{v_0\ge \al_0\}$ and $\{v_0>\al_0\}$ are convex sets, implies that $\{v_0=\al_0\}=\{v_0\ge\al_0\}\setminus \{v_0>\al_0\}$ is an $n$-dimensional convex ring/shell containing a nonempty open set. This is a contradiction since $F(D^2v_0)>0$ in $\{v_0>0\}\setminus K$.

\medskip\noindent\emph{Case} $\mathbf{B}$. Suppose $z^0\in \left(\overline{D_T^*}\setminus (K\times[0,T])\right)\cap \{t>0\}$ and $v^*(z^0)=v^{\rho_0}(z^0)>0$. Then $z^0\in (D_T^*\cap D_T^{\rho_0})\setminus (K\times[0,T])$. Since $h$ is strictly positive in $(0,\infty)$, we have $h(v^*(z^0))=h(v^{\rho_0}(z^0))>\rho_0^2h(v^{\rho_0}(z^0))$. By continuity, we further have $h(v^*)>\rho_0^2h(v^{\rho_0})$ in $Q_r(z^0)\subset (D_T^*\cap D_T^{\rho_0})\setminus (K\times(0,T])$ for some small $r>0$. It follows that
$$
F(D^2v^*)-\partial_tv^*\ge h(v^*)>\rho_0^2h(v^{\rho_0})=F(D^2v^{\rho_0})-\partial_tv^{\rho_0}\quad\text{in }Q_r(z^0),
$$
which yields
$$
\cM^-_{\la,\Lambda}(D^2(v^{\rho_0}-v^*))-\partial_t(v^{\rho_0}-v^*)\le (F(D^2v^{\rho_0})-\partial_t v^{\rho_0})-(F(D^2v^*)-\partial_tv^*)<0\quad\text{in }Q_r(z^0).
$$
On the other hand, we have $v^{\rho_0}-v^*\ge0$ in $Q_r(z^0)$ by the definition of $\rho_0$, and thus $v^{\rho_0}-v^*$ attains a local minimum at $z^0$. Therefore, by the strong minimum principle, we get $v^{\rho_0}-v^*\equiv0$ in $Q_r(z^0)$. This contradicts the strict inequality $\cM^-_{\la,\Lambda}(D^2(v^{\rho_0}-v^*))-\partial_t(v^{\rho_0}-v^*)<0$.

\medskip\noindent\emph{Case} $\mathbf{C}$. Next, we consider the case when $z^0=(x^0,t^0)\in \left(\overline{D_T^*}\setminus (K\times[0,T])\right)\cap\{t>0\}$ and $v^*(z^0)=v^{\rho_0}(z^0)=0$. Notice that $z^0\in \partial\{v^*>0\}\cap\partial\{v^{\rho_0}>0\}\cap \{t>0\}$. For $\e_1>0$ as in \eqref{eq:assump-h}, we can find by continuity small $r>0$ such that $0<v^*\le v^{\rho_0}\le \e_1$ in $Q_r(z^0)\cap D_T^*$. Then $h(v^*)=h(v^{\rho_0})=\e_1^a$ in $Q_r(z^0)\cap D^*_T$, thus we can proceed as in Case $\mathbf{B}$ to obtain 
\begin{align}\label{eq:v-v^*-supersol}
\cM^-_{\la,\Lambda}(D^2(v^{\rho_0}-v^*))-\partial_t(v^{\rho_0}-v^*)<0\quad\text{in }Q_r(z^0)\cap D_T^*.\end{align}
This implies $v^{\rho_0}>v^*$ in $Q_r(z^0)\cap D_T^*$, otherwise $v^{\rho_0}-v^*$ has a local minimum and we can argue as in Case $\mathbf{B}$ to get a contradiction.

We claim that there exist a spatial vector $\mu\in\R^n$ and a constant $\tau\ge0$ such that for any small $\eta>0$, we can construct a spatial cone $\cC^{x^0}_{\eta,\mu}:=\left\{x\in\R^n\,:\, \frac{x-x^0}{|x-x^0|}\cdot \frac{\mu}{|\mu|}>\eta \right\}$ and an oblique cylinder 
$$
\mathcal{E}_{\eta,\mu,\tau}^{z^0}:=\{(x,t)\in \R^{n+1}\,:\, t\in(t^0-1,t^0],\,x\in \cC_{\eta,\mu}^{x^0}+\tau(t^0-t)\mu\}
$$
satisfying
\begin{align}\label{eq:claim-cone}
\mathcal{E}_{\eta,\mu,\tau}^{z^0}\cap Q_r(z^0)\subset D_T^*\quad\text{for a small constant $r>0$}.
\end{align}
Suppose now the claim is true. 
We let $w$ be a solution of 
\begin{align*}
\begin{cases}
    \cM^-_{\la,\Lambda}(D^2w)-\partial_tw=0&\text{in }\cE^{z^0}_{\eta,\mu,\tau}\cap Q_r(z^0),\\
    w=0&\text{on }\partial_p \cE^{z^0}_{\eta,\mu,\tau}\cap Q_r(z^0),\\
    w=v^{\rho_0}-v^*&\text{on }\cE^{z^0}_{\eta,\mu,\tau}\cap \partial_pQ_r(z^0).
\end{cases}
\end{align*}
By using \eqref{eq:v-v^*-supersol} and that $v^{\rho_0}>v^*$ in $Q_r(z^0)\cap D^*_T$, we can apply the comparison principle to deduce $w\le v^{\rho_0}-v^*$ in $\cE^{z^0}_{\eta,\mu,\tau}\cap Q_r(z^0)$. Thus, by taking $\eta>0$ small enough, we have by Lemma~\ref{lem:par-pucci-sol-nondeg-2} in Appendix~\ref{sec:appen-nondeg} that for $s>0$ small
\begin{align}
    \label{eq:nondeg}
    \sup_{(B_s(x^0)\times\{t^0\})\cap \cE^{z^0}_{\eta,\mu,\tau}}(v^{\rho_0}-v^*)\ge \sup_{(B_s(x^0)\times\{t^0\})\cap \cE^{z^0}_{\eta,\mu,\tau}}w\ge cs^{3/2}.
\end{align}
On the other hand, we notice that $v^{\rho_0}$ solves $F(D^2v^{\rho_0})-\partial_tv^{\rho_0}=\rho_0^2h(v^{\rho_0})$ in $Q_r(z^0)\cap D_T^{\rho_0}$. Since $h=\e_1^a$ on $(0,\e_1)$ and $v^{\rho_0}(z^0)=0$, $v^{\rho_0}$ belongs to $C^{1,1}_x\cap C^{0,1}_t$ inside $Q_s(z^0)\cap D_T^{\rho_0}$ for small $s>0$, see e.g. \cite{FigSha15}. Thus, $\sup_{(B_s(x^0)\times\{t^0\})\cap \cE^{z^0}_{\eta,\mu,\tau})}v^{\rho_0}\le cs^2$ for small $s>0$, which contradicts \eqref{eq:nondeg}. 

To close the argument, we need to prove the claim \eqref{eq:claim-cone}. By the Caratheodory's theorem (see also \cite{BiaLonSal09, ColSal03}), $z^0$ can be written as a (space-time) convex combination of points $z^1,\cdots, z^k$, with $k\le n+1$, such that $z^i\in \partial\{v>0\}$, $1\le i\le k$, and there exists a hyperplane $\Pi$ in $\R^{n+1}$ supporting $\partial\{v^*>0\}$ at $z^0$ and $\partial\{v>0\}$ at $z^i$, $1\le i\le k$. Since $v^*=1$ on $\{\mathbf{0}\}\times(0,\infty)$ and $z^0\in \partial\{v^*>0\}$ with $x^0\neq\mathbf{0}$ and $t^0>0$, we have by the space-time quasiconcavity of $v^*$
$$
v^*>0\quad\text{in }\left\{(s_1 x^0,s_2 t^0)\in\R^n\times(0,\infty)\,:\, s_2>0,\, 0<s_1<\min\{s_2,1\} \right\}
$$
This implies that if $\nu=(\nu_{x^0},\nu_{t^0})$ is the unit normal to $\Pi$ pointing towards $\{v^*>0\}$ at $z^0$ and $\partial\{v>0\}$ at $z^i$, then $\nu$ cannot be parallel to $(\mathbf{0},1)$, i.e., $\nu_{x^0}\neq\mathbf{0}$. 

Next, we observe that from \eqref{eq:assump-h} and \eqref{eq:sol-const}, $v$ can be seen as a nonnegative solution of
$$
F(D^2v)-\partial_tv=\e_1^a\chi_{\{v>0\}}\quad\text{near each free boundary point on }\partial\{v>0\}.
$$
Thus $\partial\{v>0\}\cap \{t>0\}$ is locally $C^1$ in space-time near points satisfying a thickness condition on $\{v=0\}$, see e.g., \cite{FigSha15}. As $z^0\in\{t>0\}$, at least one of $z^i$'s are contained in $\{t>0\}$, say $z^1$. It then follows that $\partial\{v>0\}$ is $C^1$ in space-time near $z^1$. This allows us to construct an $n$-dimensional spatial cone $\cC_{\eta,\nu_{x^0}}^{x^1}$ such that for large $\tau_0>0$ and small $r_1>0$,
\begin{align}\label{eq:cylinder}
\cE_{\eta,\nu_{x^0},\tau_0}^{z^1}\cap Q_{r_1}(z^1)\subset \{v>0\}\subset \{v^*>0\}.
\end{align}
Here, $\tau_0$ can be chosen independently of $\eta$, as the requirement for $\tau_0$ is that the "slope" of the oblique cylinder $\cE_{\eta,\nu_{x^0},\tau_0}^{z^1}$ is less than that of the hyperplane $\Pi$. This condition can be expressed analytically as 
$$
\frac1{\tau_0|\nu_{x^0}|}<\frac{|\nu_{x^0}|}{|\nu_{t^0}|}.
$$
Recall we have showed $\nu_{x^0}\neq\mathbf{0}$.

Now, since $\{v^*>0\}$ is a convex set and $z^0$ is a convex combination of $\{z^i\}^k_{i=1}$, \eqref{eq:cylinder} implies \eqref{eq:claim-cone}.
\end{proof}

We are now ready to prove our main result, Theorem~\ref{thm:main-exterior}.

\begin{proof}[Proof of Theorem~\ref{thm:main-exterior}]
For large $i\in\mathbb{N}$, let $g^i:\R\to\R$ be a function defined by
\begin{align*}
    g^i(s):=\begin{cases}
    0,&-\infty<s\le0,\\
    (1/2)^{ia},&0<s<(1/2)^i,\\
    s^a,&s\ge(1/2)^i.
    \end{cases}
\end{align*}
Note that $g^i\nearrow g$ as $i\to\infty$ and $F(D^2u_0)\ge g(u_0)\ge g^i(u_0)$. Moreover, we consider a sequence of nonnegative compactly supported functions $u^i_\infty$ satisfying \eqref{eq:v-infty} with $g^i$ and $u_0$ in the place of $h$ and $v_0$, respectively. Since $0\le u_\infty^i\le1$ by the maximum principle, we have
$$
F(D^2u_\infty^i)=g^i(u_\infty^i)\ge1\quad\text{in } \{u_\infty^i>0\}.
$$
Thus, the non-degeneracy gives that $\supp u_\infty^i\subset B_{R_0}$ for some $R_0>0$, independent of $i$.

Now, thanks to Proposition~\ref{prop:quasiconcave-const}, for each $i$ there exists a nonnegative space-time quasiconcave function $u^i$ which satisfies the regularized problem 
\begin{align*}
    \begin{cases}
    F(D^2u^i)-\partial_tu^i=g^i(u^i)&\text{in } (\R^n\setminus K)\times(0,\infty),\\
    u^i=u_0&\text{on }\{t=0\},\\
    |\D u|=u=0&\text{on }\partial\Omega_i\cap(\R^n\times(0,\infty)),\\
    u^i=1&\text{in } K\times[0,\infty),
\end{cases}\end{align*}
where $\Omega_i=\{u^i>0\}$. Moreover, each $u^i$ is monotone in time and satisfies \eqref{eq:reg-FN-FB} and $\{u_0>0\}\times(0,\infty)\subset \Omega_i\subset B_{R_0}\times(0,\infty)$. Since $g^i\ge0$ and $u_0$ is bounded above by $1$, we have by the maximum principle that $0\le u^i\le1$ in $(\R^n\setminus K)\times(0,\infty)$. By the Arzelà-Ascoli Theorem, there exists a nonnegative continuous function $u$ in $(\R^n\setminus K)\times(0,\infty)$ such that over a subsequence $u^i\to u$ uniformly on compact subsets of $(\R^n\setminus K)\times(0,\infty)$. From the convergence $g^i\to g$ in every compact subset of $(0,\infty)$, we infer that $u$ is a space-time quasiconcave function which is nondecreasing in time and satisfies $F(D^2u)-\partial_tu=g(u)$ in $(\R^n\setminus K)\times(0,\infty)$. Clearly, $\{u>0\}\subset B_{R_0}\times(0,\infty)$. In addition, for $v^\sharp$ as in \eqref{eq:supersol} with the initial value $v_0$ replaced by $u_0$, we have $u^i\le v^\sharp$; see Step 1 in the proof of Lemma~\ref{lem:quasiconcave-const-approx}. This, along with the monotonicity of $u^i$, yields $u_0(x)\le u^i(x,t)\le v^\sharp(x,t)$ for all $(x,t)\in B_{R_0}\times(0,\infty)$. Taking $i\to\infty$ and letting $u\equiv1$ in $K^{\mathrm{o}}\times(0,\infty)$, we get $u=1$ in $K\times[0,\infty)$, $u=u_0$ on $\{t=0\}$ and $\{u_0>0\}\times(0,\infty)\subset\{u>0\}$. Finally, the growth condition \eqref{eq:reg-FN-FB} for $u^i$ and the uniform convergence $u^i\to u$ give $|\D u|=u=0$ on $\partial\Omega\cap(\R^n\times(0,\infty))$. This completes the proof.    
\end{proof}

%%%%%%%%%%%%%%%%%%%%%%%%%%%%%%%%%%%%%%%%%%%%%%%%%%%%%%%%%%%%%%%

\appendix

\section{Nondegeneracy}\label{sec:appen-nondeg}
In this section, we derive the nondegeneracy result for the parabolic equation involving Pucci operator, following the argument for the elliptic counterpart in \cite{AllKriSha23}.

We use the same notation $D_{L,R}$ as in \cite{AllKriSha23} for Lipschitz domains in $\R^n$: $$
D_{L,R}:=\{(x',x_n)\in B_R\,:\, x_n>f(x')\},
$$
where $f$ is a Lipschitz function with constant at most $L$ and satisfies $f(0)=0$.\\
For such $D_{L,R}$ and a constant $\tau\ge0$, we define $$
E_{L,R,\tau}:=\{(x,t)\in\R^{n+1}\,:\,t\in(-1,0],\,x\in D_{L,R}-\tau te_n\}.
$$

\begin{lemma}
\label{lem:par-pucci-sol-nondeg-1}
For any $\tau\ge0$, there exist small constants $\eta>0$, $\ka_*>0$ and $c_*>0$, depending only on $n,\la,\Lambda,\tau$, such that if $u$ satisfies
\begin{align}
    \label{eq:nondeg-cond}
    \begin{cases}
        (a)\, \cM^-_{\la,\Lambda}(D^2u)-\p_tu=0\quad\text{in }E_{L,1,\tau},\\
        (b)\,u\ge0\quad\text{in }E_{L,1,\tau},\\
        (c)\, u(e_n/2,-\ka_*)\ge1,\\
        (d)\,L\le\eta,
    \end{cases}
\end{align}
then $u(x,t)\ge c_*x_n$ whenever $(x,t)\in E_{L,1/16,\tau}$, $x_n\ge 2\eta$ and $-(64\eta)^2\le t\le0$.
\end{lemma}

\begin{proof}
We divide the proof of this lemma into a few of steps.

\medskip\noindent \emph{Step 1.} We define a spatial cone $\cC:=\{x\in(x',x_n)\in B_1\,:\, x_n>\frac18|x'|\}$ and an oblique cylinder $\cE_\tau:=\{(x,t)\in\R^{n+1}\,:\, t\in(-1,0],\, x\in \cC-\tau te_n\}$, and notice that $\cE_\tau\subset E_{L,1,\tau}$ whenever $L\le\eta\le 1/{16}$. We take $0<\ka<1/8$ small, depending only on $\tau$, such that 
$$
B_{3/8}(e_n/2)\times[-\ka,0]\subset \cE_\tau.
$$
By Harnack inequality (Theorem 2.4.32 in \cite{ImbSil13}), we have for some constant $c_1=c_1(n,\la,\Lambda)$
\begin{align*}
\inf_{Q_\ka(e_n/2,0)}u\ge c_1\sup\left\{u(x,t)\,:\,|x-e_n/2|<\frac\ka{2\sqrt2},\,\,-\ka+\frac38\ka^2<t<-\ka+\frac12\ka^2\right\}
\end{align*}
By taking $\ka_*:=\ka-\frac12\ka^2$ and using $(c)$ in \eqref{eq:nondeg-cond}, we further have
$$
\inf_{Q_\ka(e_n/2,0)}u\ge c_1.
$$
We let $\phi(x):=|x|^{-q}$ for some large constant $q>0$ to be determined later, and consider a function
$$
h(x,t):=c_1\frac{\left(1+\frac{1-8\ka/3}{\ka^2}t\right)^q\phi(x-e_n/2)-(3/8)^{-q}}{\ka^{-q}-(3/8)^{-q}}
$$
and a set $$
A:=\{(x,t)\in Q_1\,:\, -\ka^2\le t\le0,\,\, 0\le h(x,t)\le c_1\}.
$$
Here, the function $h$ and the set $A$ are constructed so that \begin{align}
    \label{eq:u-h-bdry-cond}
    A\Subset E_{L,1,\tau}\quad\text{and}\quad u\ge h\text{ on }\p_pA.
\end{align}
Indeed, one can easily see that $$
A=\left\{(x,t)\,:\,-\ka^2\le t\le0,\,\,\ka\left(1+\frac{1-\frac{8\ka}3}{\ka^2}t\right)\le|x-e_n/2|\le3/8\left(1+\frac{1-\frac{8\ka}3}{\ka^2}t\right)\right\}.
$$
Thus, each time slice of $A$ at time $t$, $-\ka^2\le t\le0$, is an annulus centered at $e_n/2$, and the annuli shrink as time $t$ decreases. In particular, \begin{align*}
    &A\cap\{t=0\}=\left(B_{3/8}(e_n/2)\setminus B_\ka(e_n/2)\right)\times\{0\},\\
    &A\cap\{t=-\ka^2\}=\left(B_{\ka}(e_n/2)\setminus B_{8\ka^2/3}(e_n/2)\right)\times\{-\ka^2\}.
\end{align*}
This readily gives
$$
A\subset B_{3/8}(e_n/2)\times[-\ka^2,0]\Subset \cE_\tau\subset E_{L,1,\tau}.
$$
In addition, the lateral boundary $\p_p A\cap\{-\ka^2<t\le0\}$ consists of the following two surfaces: \begin{align*}
    &\text{the outer surface }S_1=\{(x,t)\in Q_1\,:\, -\ka^2<t\le0,\,\, h(x,t)=0\},\\
    &\text{the inner surface }S_2=\{(x,t)\in Q_1\,:\, -\ka^2<t\le0,\,\, h(x,t)=c_1\}.
\end{align*}
Notice that both $S_1$ and $S_2$ are lateral boundaries of conical frustums. The parabolic boundary $\partial_pA$ is composed of $S_1$, $S_2$ and the bottom $\partial_pA\cap\{t=-\ka^2\}$. On $S_1$, we simply have $h=0\le u$. Regarding $S_2$ and $\partial_pA\cap \{t=-\ka^2\}$, as they are contained in $\overline{Q_\ka(e_n/2,0)}$ and $h\le c_1$ in $A$, we have $h\le c_1\le u$ on $S_2\cup(\partial_pA\cap\{t=-\ka^2\})$. Hence, \eqref{eq:u-h-bdry-cond} holds.\\
\indent Next, we show that $h$ is a subsolution in $A$. For this purpose, we compute that in $B_1\setminus\{0\}$
$$
D^2\phi(x)=q|x|^{-q-2}\left(\frac{q+2}{|x|^2}x\otimes x-I_n\right),
$$
and thus \begin{align*}
    \cM_{\la,\Lambda}^-(D^2\phi)&\ge q|x|^{-q-2}\left(\frac{q+2}{|x|^2}\cM_{\la,\Lambda}^-(x\otimes x)+\cM_{\la,\Lambda}^-(-I_n)\right)\\
    &=q|x|^{-q-2}\left(\la(q+2)-\Lambda n\right).
\end{align*}
It follows that in $A$
\begin{align*}
    &\cM^-_{\la,\Lambda}(D^2h)-\p_th\\
    &\qquad=c_1\frac{\left(1+\frac{1-8\ka/3}{\ka^2}t\right)^q}{\ka^{-q}-(3/8)^{-q}}\cM^-_{\la,\Lambda}(D^2\phi(x-e_n/2))-\p_th\\
    &\qquad\ge c_1\frac{\left(1+\frac{1-8\ka/3}{\ka^2}t\right)^q}{\ka^{-q}-(3/8)^{-q}}q|x-e_n/2|^{-q-2}\left(\la(q+2)-\Lambda n\right)\\
    &\qquad\qquad-c_1\frac{\phi(x-e_n/2)}{\ka^{-q}-(3/8)^{-q}}q\left(1+\frac{1-8\ka/3}{\ka^2}t\right)^{q-1}\left(\frac{1-8\ka/3}{\ka^2}\right)\\
    &\qquad=\frac{c_1q|x-e_n/2|^{-q-2}}{\ka^{-q}-(3/8)^{-q}}\left(1+\frac{1-8\ka/3}{\ka^2}t\right)^{q-1}\times\\
    &\qquad\qquad \times\left[\left(1+\frac{1-8\ka/3}{\ka^2}t\right)\left(\la(q+2)-\Lambda n\right)-|x-e_n/2|^2\left(\frac{1-8\ka/3}{\ka^2}\right)\right].
\end{align*}
Here, if $q$ is large, then we have \begin{align*}
    &\left(1+\frac{1-8\ka/3}{\ka^2}t\right)\left(\la(q+2)-\Lambda n\right)-|x-e_n/2|^2\left(\frac{1-8\ka/3}{\ka^2}\right)\\
    &\qquad\ge \left(\frac{8\ka}3\right)\left(\la(q+2)-\Lambda n\right)-\frac{(3/8)^2}{\ka^2}>0,
\end{align*}
thus \begin{align*}
    \cM^-_{\la,\Lambda}(D^2h)-\p_th\ge 0=\cM^-_{\la,\Lambda}(D^2u)-\p_tu\quad\text{in }A.
\end{align*}
With this and \eqref{eq:u-h-bdry-cond} at hand, we can apply the comparison principle and get $$
u\ge h\quad\text{in }A.
$$

\medskip\noindent \emph{Step 2.} 
Let 
$$
c_0:=c_1\frac{(11/{32})^{-q}-(3/8)^{-q}}{\ka^{-q}-(3/8)^{-q}}\in (0,c_1).
$$
Then $A_0:=\{(x,t)\in Q_1\,:\,-\ka^2\le t\le 0,\, c_0<h(x,t)<c_1\}$ satisfies the following:
\begin{align*}
    A_0\subset A,\quad u\ge h>c_0 \text{ in }A_0,\quad A_0\cap\{t=0\}=\left(B_{11/32}(e_n/2)\setminus B_\ka(e_n/2)\right)\times\{t=0\}.
\end{align*}
Fix a point $(y,0)=(y',y_n,0)$ with $|y'|\le1/16$ and $y_n=1/4$. For $r:=1/4-\eta$, we define $$h_y(x,t):=c_0\frac{\left(1+\frac{1-\tilde\ka/r}{\tilde\ka^2}t\right)^{\tilde q}\tilde\phi(x-y)-r^{-\tilde q}}{\tilde\ka^{-\tilde q}-r^{-\tilde{q}}}
$$
in $$
A_y:=\{(x,t)\in Q_1\,:\, -\tilde\ka^2\le t\le 0,\,\,0<h_y(x,t)<c_0\}.
$$
Here, $\tilde\ka$ and $\tilde q$ are small constants, depending only on $n,\la,\Lambda,\tau$, to be determined later, and $\tilde\phi(x):=|x|^{-\tilde q}$. Then, \begin{align*}
    A_y&\subset\{(x,t)\in Q_1\,:\,-\tilde\ka^2\le t\le 0,\,\,0<h_y(x,t)\}\\
    &=\{(x,t)\in Q_1\,:\,-\tilde\ka^2\le t\le 0,\,\,|x-y|<r+\frac{r-\tilde\ka}{\tilde\ka^2}t\}\\
    &\subset \{(x,t)\in Q_1\,:\,-\tilde\ka^2\le t\le 0,\,\,|x-y|<r+\frac{t}{\tilde\ka}\}.
\end{align*}
This, along with the fact that $r=1/4-\eta<|y|$, implies that $A_y\Subset E_{L,1,\tau}$ for small $\tilde\ka>0$. Moreover, as before, the lateral boundary $\p_p A_y\cap\{-\tilde\ka^2<t\le0\}$ consists of two conical cylinder-shaped surfaces. On the outer surface, we simply have $h_y=0\le u$. Concerning the inner  surface $S_y:=\{(x,t)\in Q_1\,:\, -\tilde\ka^2\le t\le 0,\, h_y(x,t)=c_0\}$, we observe that
$$
S_y\cap\{t=0\}=B_{\tilde\ka}(y)\times\{0\}\subset\left(B_{11/32}(e_n/2)\setminus B_\ka(e_n/2)\right)\times\{0\}=A_0\cap \{t=0\},
$$
where the inclusion holds whenever $\ka+\tilde \ka<1/8$. Thus, if $\tilde\ka$ is small enough, then we have $S_y\Subset A_0$, and thus $h_y=c_0\le u$ in $S_y$. Similarly, from $\partial_pA_y\cap\{t=-\tilde\ka^2\}=\left(B_{\tilde\ka}(y)\setminus B_{\tilde\ka^2/r}(y)\right)\times\{-\tilde\ka^2\}$ and $B_{\tilde\ka}(y)\subset B_{11/32}(e_n/2)\setminus B_\ka(e_n/2)$, we infer that $\partial_pA_y\cap\{t=-\tilde\ka^2\}\subset A_0$ for $\tilde\ka$ small enough, and hence $h_y\le c_0\le u$ on $\partial_pA_y\cap\{t=-\tilde\ka^2\}$. Therefore, we conclude that
 $$
h_y\le u\quad\text{on }\p_pA_y.
$$
In addition, we can compute (as we have done for $\cM_{\la,\Lambda}^-(D^2h)-\p_th$) \begin{align*}
    \cM_{\la,\Lambda}^-(D^2h_y)-\p_th_y&\ge \frac{c_0\tilde q|x-y|^{-\tilde q-2}}{\tilde\ka^{-\tilde q}-r^{-\tilde q}}\left(1+\frac{1-\tilde\ka/r}{\tilde\ka^2}t\right)^{q-1}\times\\
    &\qquad\times\left[\left(1+\frac{1-\tilde\ka/r}{\tilde\ka^2}t\right)\left(\la(\tilde q+2)-\Lambda n\right)-|x-y|^2\frac{1-\tilde\ka/r}{\tilde\ka^2}\right]
\end{align*}
in $A_y$. Using that $r=1/4-\eta\in (1/8,1/4)$, we can obtain that $\cM_{\la,\Lambda}^-(D^2h_y)-\p_th_y\ge 0$ in $A_y$ if $\tilde q=\tilde q(n,\la,\Lambda,\tilde\ka)$ is large enough. Thus, by the comparison principle, $$
u\ge h_y\quad\text{in }A_y.
$$

\medskip\noindent \emph{Step 3.} Now, we are ready to prove Lemma~\ref{lem:par-pucci-sol-nondeg-1}. For $y$ as above (i.e., $|y'|\le 1/16$ and $y_n=1/4$), let $(x,t)$ be a point satisfying $x'=y'$, $2\eta<x_n<1/4-\tilde\ka$ and $-(64\eta)^2\le t\le0$. Then $(x,t)\in A_y$ for small $\eta>0$. Thus \begin{align*}
    u(x,t)&\ge h_y(x,t)\\
    &=\frac{c_0}{\tilde\ka^{-\tilde q}-r^{-\tilde q}}\left[\left(1+\frac{1-\tilde\ka/r}{\tilde\ka^2}t\right)^{\tilde q}|x-y|^{-\tilde q}-r^{-\tilde q}\right]\\
    &\ge c_0\tilde\ka^{\tilde q}\left[\left(1+\frac{1-\tilde\ka/r}{\tilde\ka^2}t\right)^{\tilde q}|x-y|^{-\tilde q}-r^{-\tilde q}\right]\\
    &= c_0\tilde\ka^{\tilde q}\left(|x-y|^{-\tilde q}-r^{-\tilde q}\right)-c_0\tilde\ka^{\tilde q}\left(1-\left(1+\frac{1-\tilde\ka/r}{\tilde\ka^2}t\right)^{\tilde q}\right)|x-y|^{-\tilde q}\\
    &=:I-II.
\end{align*}
Using $x_n\ge2\eta$, we get 
\begin{align*}
    I&=c_0\tilde\ka^{\tilde q}\left((1/4-x_n)^{-\tilde q}-(1/4-\eta)^{-\tilde q}\right)\ge c(n,\la,\Lambda,\tau)(x_n-\eta)\ge c_2x_n
\end{align*}
for some $c_2=c_2(n,\la,\Lambda,\tau)>0$. On the other hand, to bound $II$, we observe that for $\eta$ small
$$
1-\left(1+\frac{1-\tilde\ka/r}{\tilde\ka^2}t\right)^{\tilde q}\le 1-\left(1+\frac{1-\tilde\ka/r}{\tilde\ka^2}(64\eta)^2\right)^{\tilde q}\le C\eta^2\le C\eta x_n.
$$
From $|x'|=|y'|$ and $x_n<1/4-\tilde\ka$, we also have
$$
|x-y|^{-\tilde q}=|1/4-x_n|^{-\tilde q}\le \tilde\ka^{-\tilde q}\le C(n,\la,\Lambda,\tau).
$$
Thus, $$
II\le C\eta x_n\le c_2/2\,x_n,
$$
where the last inequality holds if we take $\eta$ small enough. Therefore, we conclude by taking $c_*:=c_2/2$
$$
u(x,t)\ge I-II\ge c_*x_n.
$$
This completes the proof.
\end{proof}

\begin{lemma}
\label{lem:par-pucci-sol-nondeg-2}
Let $\tau\ge0$ and $1<\gamma<2$ be given. Then, there exist small constants $\ka_*>0$ and $\eta>0$, depending only on $n,\la,\Lambda,\tau,\gamma$, such that if $u$ solves \eqref{eq:nondeg-cond}, then 
$$
u(x,0)\ge c(n,\la,\Lambda,\gamma)\dist(x,\p D_{L,1})^\gamma\quad\text{for all }x\in D_{L,1/64}.
$$
\end{lemma}

\begin{proof}
Let $c_*=c_*(n,\la,\Lambda,\tau)>0$ be as in Lemma~\ref{lem:par-pucci-sol-nondeg-1}. One can easily check in its proof that $c_*$ was chosen independently of $\eta$. We let $r_1:=64\eta$ for small $\eta>0$ to be determined later. Particularly, we ask $\eta$ to be smaller than the one in Lemma~\ref{lem:par-pucci-sol-nondeg-1}. We claim that if $x=(x',x_n)\in\p D_{L,1/64}$, then \begin{align}
    \label{eq:nondeg-induc}
    u(x',\rho,s)\ge\left(\frac{c_*}2\right)^{k+1}(\rho-x_n)
\end{align}
whenever $x_n+2\eta r_1^k\le\rho\le x_n+r_1^k/16$ and $-r_1^{2(k+1)}\le s\le0$ for some $k\in\N$. Clearly, $(x',\rho,s)\in E_{L,1/16,\tau}$ holds for small $r_1$ (or, equivalently, for small $\eta$). We prove the claim by induction on $k\in\N$. 

For $k=1$, we let $\ka_*=\ka_*(n,\la,\Lambda,\tau)>0$ be as in Lemma~\ref{lem:par-pucci-sol-nondeg-1} and consider 
$$
u_1(y,t):=\frac{u(x+r_1y,r_1^2t)}{u(x+\frac{r_1}2e_n,-r_1^2\ka_*)},\quad (y,t)\in E_{L,1,\tau}.
$$
Here, the set $E_{L,1,\tau}$ may be equipped with a different Lipschitz function $f$ from the one in \eqref{eq:nondeg-cond}. Clearly, $u_1$ satisfies the conditions $(a)-(d)$ in \eqref{eq:nondeg-cond}. Applying Lemma~\ref{lem:par-pucci-sol-nondeg-1} to $u$ and $u_1$ gives 
\begin{align*}
    &u\left(x+\frac{r_1}2e_n,-r_1^2\ka_*\right)\ge c_*\left(x_n+\frac{r_1}2\right)\ge\frac{c_*}2r_1,\\
    &u_1(0,y_n,t)\ge c_*y_n\quad\text{for }2\eta\le y_n\le 1/16,\,-r_1^2\le t\le0.
\end{align*}
By putting $\rho=x_n+r_1y_n$ and $s=r_1^2t$, we obtain that \begin{align*}
    u(x',\rho,s)&=u(x',x_n+r_1y_n,r_1^2t)=u_1(0,y_n,t)u\left(x+\frac{r_1}2e_n,-r_1^2\ka_*\right)\\
    &\ge\frac{c_*^2}2(r_1y_n)\ge\left(\frac{c_*}2\right)^2(\rho-x_n)
\end{align*}
for $x_n+2\eta r_1\le\rho\le x_n+r_1/16$ and $-r_1^4\le s\le 0$. This proves \eqref{eq:nondeg-induc} for $k=1$.

Now we suppose that \eqref{eq:nondeg-induc} holds for $k-1$, and prove it for $k$. For this aim, we define $$
u_k(y,t):=\frac{u(x+r_1^ky,r_1^{2k}t)}{u\left(x+\frac{r_1^k}2e_n,-r_1^{2k}\ka_*\right)},\quad(y,t)\in E_{L,1,\tau}.
$$
From $2\eta r_1^{k-1}\le\frac{r_1^k}2\le\frac{r_1^{k-1}}{16}$, we have by the induction hypothesis 
$$
u\left(x+\frac{r_1^k}2e_n,-r_1^{2k}\ka_*\right)\ge\left(\frac{c_*}2\right)^k\left(\frac{r_1^k}2\right).
$$
Since $u_k$ satisfies $(a)-(d)$ in \eqref{eq:nondeg-cond}, we can apply Lemma~\ref{lem:par-pucci-sol-nondeg-1} to obtain that 
$$
u_k(0,y_n,t)\ge c_*y_n\quad\text{for }2\eta\le y_n\le 1/16,\,-r_1^2\le t\le 0.
$$
By putting $\rho=x_n+r_1^ky_n$ and $s=r_1^{2k}t$, we deduce that \begin{align*}
    u(x',\rho,s)&=u(x',x_n+r_1^ky_n,r_1^{2k}t)=u_k(0,y_n,t)u\left(x+\frac{r_1^k}2e_n,-r_1^{2k}\ka_*\right)\\
    &\ge\left(\frac{c_*}2\right)^{k+1}(r_1^ky_n)=\left(\frac{c_*}2\right)^{k+1}(\rho-x_n).
\end{align*}
This finishes the proof for the claim \eqref{eq:nondeg-induc}.

Now, we are ready to prove Lemma~\ref{lem:par-pucci-sol-nondeg-2}. To this end, we fix $x=(x',x_n)\in\p D_{L,1/64}$ and let $\rho\in(x_n,2\eta)$. Then we can take $k\in\N$ such that $2\eta r_1^k\le\rho-x_n\le r_1^k/16$. By \eqref{eq:nondeg-induc}, $$
u(x',\rho,0)\ge\left(\frac{c_*}2\right)^{k+1}(\rho-x_n).
$$
From $r_1^k\le\frac{\rho-x_n}{2\eta}$, we have$$
\left(\frac{c_*}2\right)^k\ge\left(\frac{c_*}2\right)^{\log_{r_1}\left(\frac{\rho-x_n}{2\eta}\right)}=\left(\frac{\rho-x_n}{2\eta}\right)^{\log_{r_1}\left(\frac{c_*}2\right)}
$$
If $\eta=\eta(n,\la,\Lambda,\tau,\gamma)$ is small enough, then
$$
\log_{r_1}\left(\frac{c_*}2\right)=\frac{\log\left(\frac{c_*}2\right)}{\log(64\eta)}\in(0,\gamma-1).
$$
Thus, we have for $x_n<\rho<2\eta$\begin{align*}
    u(x',\rho,0)&\ge\left(\frac{c_*}2\right)^k\frac{c_*}2(\rho-x_n)\ge\left(\frac{\rho-x_n}{2\eta}\right)^{\gamma-1}\frac{c_*}2(\rho-x_n)\\
    &\ge c(n,\la,\Lambda,\tau,\gamma)(\rho-x_n)^\gamma.
\end{align*}
On the other hand, when $(x',\rho)\in D_{L,1/64}$ with $\rho\ge2\eta$, we can simply use Lemma~\ref{lem:par-pucci-sol-nondeg-1} to obtain $$
u(x',\rho,0)\ge c_*\rho\ge c_*(\rho-x_n)^\gamma.$$
Therefore, we conclude that if $(x',\rho)\in D_{L,1/64}$ (and $x=(x',x_n)\in \p D_{L,1/64}$), then $$
u(x',\rho,0)\ge c(\rho-x_n)^\gamma\ge c\dist((x',\rho),\p D_{L,1})^\gamma.
$$
This completes the proof.
\end{proof}

\section*{Declarations}

\noindent {\bf  Data availability statement:} All data needed are contained in the manuscript.

\medskip
\noindent {\bf  Funding and/or Conflicts of interests/Competing interests:} The authors declare that there are no financial, competing or conflict of interests.

%%%%%%%%%%%%%%%%%%%%%%%%%%%%%%%%%%%%%%%%%%%%%%%%%%%

\end{document}